\documentclass[11pt,a4paper]{amsart}
\usepackage{amssymb}
\usepackage{mathabx}
\usepackage{mathrsfs}
\usepackage{syntonly}
\usepackage{amsmath}
\usepackage{amsthm}
\usepackage{amsfonts}
\usepackage{amssymb}
\usepackage{latexsym}
\usepackage{amscd,amssymb,amsopn,amsmath,amsthm,graphics,amsfonts,mathrsfs,accents,enumerate,verbatim,calc}
\usepackage[dvips]{graphicx}
\usepackage[colorlinks=true,linkcolor=red,citecolor=blue]{hyperref}
\usepackage[all]{xy}

\date{}
\allowdisplaybreaks[4] \footskip=15pt
\renewcommand{\uppercasenonmath}[1]{}

\numberwithin{equation}{section} \theoremstyle{plain}
\newtheorem{lem}{Lemma}[section]
\newtheorem{cor}[lem]{Corollary}
\newtheorem{prop}[lem]{Proposition}
\newtheorem{thm}[lem]{Theorem}

\newtheorem{definition}[lem]{Definition}
\newtheorem{Ex}[lem]{Example}
\newtheorem{Quest}[lem]{Question}
\newtheorem{Property}[lem]{Property}
\newtheorem{Properties}[lem]{Properties}
\newtheorem{Subprops}{}[lem]
\newtheorem{Para}[lem]{}

\newtheorem{remark}[lem]{Remark}
\newtheorem{rem}[lem]{Remark}

\newenvironment{df}{\begin{definition}\rm}{\end{definition}}
\newenvironment{ex}{\begin{Ex}\rm}{\end{Ex}}

\newtheorem*{ack*}{ACKNOWLEDGEMENTS}

\newcommand{\pf}{\noindent\begin {proof}}
\newcommand{\epf}{\end{proof}}

\newcommand{\X}{\mathcal{X}}
\newcommand{\W}{\mathcal{W}}

\newcommand{\C}{\mathcal{C}}
\newcommand{\HH}{\mathcal{H}}

\newcommand{\E}{\mathbb{E}}

\newcommand{\w}{\widecheck}

\newcommand{\ra}{\rightarrow}

\begin{document}
\begin{center}
{\large  \bf   A new characterization of silting subcategories in the stable category of a Frobenius extriangulated category}

\vspace{0.5cm}  Yajun Ma, Nanqing Ding, Yafeng Zhang, Jiangsheng Hu

\end{center}

\bigskip
\centerline { \bf  Abstract}
\medskip

\leftskip10truemm \rightskip10truemm \noindent We give a new characterization of silting subcategories in the stable category of a Frobenius extriangulated category, generalizing the result of Di et al. (J. Algebra 525 (2019) 42-63) about the Auslander-Reiten type correspondence for silting subcategories over triangulated categories. More specifically, for any Frobenius extriangulated category $\mathcal{C}$, we establish a bijective correspondence
between silting subcategories of the stable category $\underline{\mathcal{C}}$ and certain covariantly finite subcategories of $\mathcal{C}$. As a consequence, a characterization of
silting subcategories in the stable category of a Frobenius exact category is given. This result is applied to homotopy categories over abelian categories with enough projectives, derived categories over Grothendieck categories with enough projectives as well as to the stable category of Gorenstein projective modules over a ring $R$.\\[2mm]

\leftskip0truemm \rightskip0truemm
\section { \bf Introduction}

Silting subcategories and objects in triangulated categories were introduced by Keller and Vossieck \cite{KV} and later studied by Aihara and Iyama \cite{AI}, as a way of overcoming a problem inherent to tilting objects, namely, that mutations are sometimes impossible to define.
According to \cite{IY,NSZ}, silting subcategories (objects) are closely related to t-structures and co-t-structures in triangulated categories. Thus one can better understand the structure of a triangulated category from the viewpoint of the silting theory.
Recently, Di, Liu, Wang and Wei \cite{DLWW} established a bijective correspondence
between silting subcategories of a triangulated category $\mathcal{T}$ and certain covariantly finite subcategories of $\mathcal{T}$. The notion of covariantly finite subcategories was first introduced by Auslander and Smal${\o}$ \cite{AS1980} (note that a subcategory is called covariantly finite precisely when it is preenveloping in the sense of Enochs and Jenda \cite{EJ2}). It should be noted that the bijective correspondence mentioned above is known as the ``Auslander-Reiten type correspondence" for silting subcategories in \cite{DLWW}, which is inspired by the classical Auslander-Reiten type correspondence for tilting modules established by Auslander and Reiten in \cite{AR}.

Happel \cite{Happel1988} showed that if $(\mathcal{C}, \mathcal{S})$ is a Frobenius exact category, then its stable category $\underline{\mathcal{C}}$ carries a triangulated structure. Beligiannis obtained a similar result \cite[Theorem 7.2]{Bel1} by replacing $\mathcal{C}$ with a triangulated category $\mathcal{T}$ and replacing $\mathcal{S}$ with a proper class of triangles $\mathcal{E}$. Recently, Nakaoka and Palu unified  the above results and proved
that the stable category of a Frobenius extriangulated category is a triangulated category
(see \cite[Corollary 7.4 and Remark 7.5]{NP}). The notion of extriangulated categories was introduced by Nakaoka and Palu in \cite{NP} as a simultaneous generalization of
exact categories and triangulated categories. As typical examples we
have that Frobenius exact categories and triangulated categories are Frobenius extriangulated categories (see \cite[Example 7.2(1) and Corollary 7.6]{NP}). However, there exist some  examples of Frobenius extriangulated categories which are neither exact nor
triangulated (see \cite[Example 4.14]{ZZ}).

The aim of this paper is to give a new characterization of silting subcategories in the stable category of a Frobenius extriangulated category. More precisely, for any Frobenius extriangulated category $\mathcal{C}$, we give a bijective correspondence
between silting subcategories of the stable category $\underline{\mathcal{C}}$ and certain covariantly finite subcategories of $\mathcal{C}$.

In dealing with the above problem, the main technical problem we encounter is that one has to choose an appropriate covariantly finite subcategory in a Frobenius extriangulated category $\mathcal{C}$ from a given silting subcategory in the stable category $\underline{\mathcal{C}}$. To circumvent this problem, by the bijective correspondence between silting subcategories and bounded co-t-structures in a triangulated category established by Mendoza Hern$\acute{\mathrm{a}}$ndez, S$\acute{\mathrm{a}}$enz, Santiago Vargas and Souto Salorio in \cite[Corollary 5.9]{MSSS2}, we first provide an explicit procedure to construct cotorsion pairs in the ``new" extriangulated category induced by the certain subcategory in a given extriangulated category, and then demonstrate that these cotorsion pairs not only give rise to constructing bounded co-t-structures for silting subcategories in the stable category of a Frobenius extriangulated category, but also are helpful to choose certain covariantly finite subcategories in a Frobenius extriangulated category.

To state our main results more precisely, let us first introduce some definitions.

Assume that $(\mathcal{C}, \mathbb{E}, \mathfrak{s})$ is an extriangulated category, where $\C$ is an additive category, $\mathbb{E}: \C^{\rm op}\times \C \rightarrow {\rm Ab}$ is an additive bifunctor and $\mathfrak{s}$ assigns to each $\delta\in \mathbb{E}(C,A)$ a class of $3$-term sequences with end terms $A$ and $C$ such that certain axioms hold (see \cite[Definition 2.10]{NP}). If there is no confusion, we also write $\mathcal{C}:=(\mathcal{C}, \mathbb{E}, \mathfrak{s})$.
We will use the following notation.

$\bullet$ A sequence $\xymatrix@C=0.6cm{A \ar[r]^{x} & B \ar[r]^{y} & C }$ is called a \emph{conflation} if it realizes some $\mathbb{E}$-extension $\delta\in\mathbb{E}(C, A)$.
In this case, $x$ is called an {\it inflation}, $y$ is called a {\it deflation}, and we write it as $$\xymatrix{A\ar[r]^x&B\ar[r]^{y}&C\ar@{-->}[r]^{\delta}&.}$$
We usually do not write this ``$\delta$" if it is not used in the argument.

$\bullet$ An object $P\in\mathcal{C}$ is called \emph{projective} if $\E(P, \C)=0$. We say that $\C$ has \emph{enough projectives} if any object $C\in{\C}$ admits a deflation $P\rightarrow C$ for some projective object $P$. Dually, we have the notions of injectives and {enough injectives}. In the following, we denote by $Proj(\C)$ the subcategory of projective objects and by $Inj(\C)$ the subcategory of injective objects. It should be noted that $Proj(\mathcal{C})$ and $Inj(\mathcal{C})$ consist of zero objects whenever $\mathcal{C}$ is a triangulated category.

$\bullet$ $\mathcal{C}$ is said to be \emph{Frobenius} if $\mathcal{C}$ has enough projectives and injectives and if moreover the projectives coincide with the injectives. In this case, one has the quotient category $\underline{\mathcal{C}}$ of $\mathcal{C}$ by projectives, which is a triangulated category by \cite[Corollary 7.4 and Remark 7.5]{NP}. We refer to this category as the \emph{stable category} of $\C$.

$\bullet$ A subcategory $\mathcal{H}$ of $\mathcal{C}$ is called \emph{extension-closed} if $\mathcal{H}$ is closed under extensions, i.e., for any $\mathbb{E}$-triangle $\xymatrix{A\ar[r]^{x}&B\ar[r]^{y}&C\ar@{-->}[r]^{\delta}&}$ with $A,C\in\mathcal{H}$, we have $B\in\mathcal{H}$. Given an $\mathbb{E}$-triangle $\xymatrix{A \ar[r]^{x} & B \ar[r]^{y} & C \ar@{-->}[r]^{\delta}&,}$ we call $A$ the \emph{cocone} of $y:B\rightarrow C$ and $C$ the \emph{cone} of $x:A\rightarrow B$.
Recall from \cite{ZZ1} that a subcategory $\mathcal{H}\subseteq \C$ is called \emph{resolving} (resp., \emph{coresolving}) if it contains $Proj(\C)$ ~(resp., $Inj(\C)$), closed under extensions, and cocones of deflations (resp., cones of inflations).

$\bullet$ Recall from \cite{ZZ1} that an \emph{$\E$-triangle sequence} in $\mathcal{C}$ is displayed as a sequence
$$\xymatrix{\cdots\ar[r]&X_{n+1}\ar[r]^{d_{n+1}}&X_{n}\ar[r]^{d_{n}}&X_{n-1}\ar[r]&\cdots&}$$
over $\mathcal{C}$ such that for any $n$, there are $\E$-triangles $\xymatrix{K_{n+1}\ar[r]^{g_{n}}&X_{n}\ar[r]^{f_{n}}&K_{n}\ar@{-->}[r]^{\delta^{n}}&}$ and the differential $d_{n}=g_{n-1}f_{n}$.

Let $\mathcal{X}$ be a subcategory of $\mathcal{C}$. We denote by $\w{\mathcal{X}}$ the class of objects $C\in \C$ such that there exists an $\E$-triangle sequence ~$C\rightarrow X_{0}\rightarrow\cdots\rightarrow X_{-n+1}\rightarrow X_{-n}$ with each $X_{-i}\in \mathcal{X}$ for some non-negative integer $n$.

%
%
%

Let $\mathcal{Y}$ be a subcategory of $\C$. Recall that a morphism $f:C\rightarrow Y$ with $Y\in{\mathcal{Y}}$ is called a \emph{left $\mathcal{Y}$-approximation}
(or a \emph{$\mathcal{Y}$-preenvelope}) of $C$ if $\C(f,Y'):\C(Y,Y')\rightarrow \C(C,Y')$ is surjective for any object $Y'\in{\mathcal{Y}}$. If any $C\in \mathcal{C}$ admits a  left $\mathcal{Y}$-approximation, then $\mathcal{Y}$ is called \emph{covariantly finite} in $\C$. For more details, we refer to \cite{AR,EJ2}. Moreover, $\mathcal{Y}$ is called \emph{specially covariantly finite} in $\C$ provided that for any $C\in \C$, there is an $\E$-triangle  $\xymatrix@C=0.6cm{C\ar[r]^{f}&Y\ar[r]^{g}&K\ar@{-->}[r]^{}&}$ such that $Y\in \mathcal{Y}$ and $\E(K,\mathcal{Y})=0$. In this case, $f:C\rightarrow Y$ is a left $\mathcal{Y}$-approximation.

If $\mathcal{C}$ has enough projectives and injectives, for any subcategory $\mathcal{M}$ of $\mathcal{C}$, we set $$\mathcal{M}^{\perp}=\{Y\in\mathcal{C} \ | \ \E^{i}(X, Y)=0 \mathrm{~for~all~} i\geq 1, \mathrm{~and~all~} X \in \mathcal{M}\},$$
where $\E^{i}(X, Y)$ is the higher extension group defined by Liu and Nakaoka in \cite[Section 5.2]{LN}.

Let $\mathcal{S}$ be a subcategory of a triangulated category $\mathcal{T}$. Following \cite[Definition 2.1]{AI}, $\mathcal{S}$ is called \emph{silting} if ${\rm{Hom}}_{\mathcal{T}}(\mathcal{S}, \mathcal{S}[i])=0$ for all $i>0$, and $\mathcal{T}={\rm thick}(\mathcal{S})$ is the smallest triangulated subcategory of $\mathcal{T}$ containing $\mathcal{S}$ and closed under direct summands and isomorphisms.

Now, our main result can be stated as follows.

\begin{thm}\label{thm5}
Let $\C$ be a Frobenius extriangulated category. The assignments
\begin{center}
 $\Psi:\underline{\mathcal{M}}\mapsto$~${\mathcal{M}^{\bot}}$~$~~$ and $~~$~$\Phi:\mathcal{H}\mapsto\underline{\mathcal{H}\bigcap{^{\bot}\mathcal{H}}}$
\end{center}
  give mutually inverse bijections between the following classes: 
\begin{enumerate}
\item Silting subcategories $\underline{\mathcal{M}}$ of the stable category $\underline{\C}$ with $Proj(\C)\subseteq \mathcal{M}$.
\item Subcategories $\mathcal{H}$ of $\C$, which are specially covariantly finite and coresolving in $\C$ such that $\w{\mathcal{H}}=\C$ and for any $H\in \mathcal{H}$, there exists a positive integer $t\geq 1$ making $\E^{i}(H,\mathcal{H})=0$ for all $i\geq t$.
\end{enumerate}
\end{thm}

If $\C$ is a triangulated category in Theorem \ref{thm5}, then we recover the Auslander-Reiten type correspondence for silting subcategories over triangulated categories established by Di, Liu, Wang and Wei in \cite{DLWW} (see Corollary \ref{cor:4.13}). Now, specializing $\C$ to a Frobenius exact category, we then get the following corollary.

\begin{cor}\label{cor:1.2} Let $\C$ be a Frobenius exact category. The assignments
\begin{center}
 $\underline{\mathcal{M}}\mapsto$~${\mathcal{M}^{\bot}}$~$~~$ and $~~$~$\mathcal{H}\mapsto\underline{\mathcal{H}\bigcap{^{\bot}\mathcal{H}}}$
\end{center}
  give mutually inverse bijections between the following classes: 
\begin{enumerate}
\item Silting subcategories $\underline{\mathcal{M}}$ of $\underline{\C}$ with $Proj(\C)\subseteq \mathcal{M}$.
\item Subcategories $\mathcal{H}$ of $\C$, which are specially covariantly finite and coresolving in $\C$ such that $\w{\mathcal{H}}=\C$ and for any $H\in \mathcal{H}$, there exists a positive integer $t\geq 1$ making $\mathrm{Ext}_{\C}^{i}(H,\mathcal{H})=0$ for all $i\geq t$.
\end{enumerate}
\end{cor}

We will apply Corollary \ref{cor:1.2} to homotopy categories over abelian categories with enough projectives, derived categories over Grothendieck categories with enough projectives as well as to the stable category of Gorenstein projective modules over a ring $R$.

%
%

The contents of this paper are outlined as follows. In Section \ref{Preliminaries}, we fix notations and recall some definitions and basic facts used throughout the paper. In Section
\ref{coresolution-dimension}, we study the coresolution dimension relative to different subcategories in $\mathcal{C}$. In Section \ref{proof-of-main-result}, we first construct cotorsion pairs for an extriangulated category (see Corollary \ref{corollary:4.4}). This is based on the homological properties of the coresolution dimension established in Section \ref{coresolution-dimension}. As a result, we give the proof of Theorem \ref{thm5}. As a consequence, we re-obtain a main result in \cite{DLWW} (see Corollary \ref{cor:4.13}). Furthermore, we deduce a series of consequences of Theorem \ref{thm5} for some different Frobenius exact categories, including Corollary \ref{cor:1.2}.

\section{\bf Preliminaries}\label{Preliminaries}
Throughout this paper, by the term $``subcategory"$ we always mean a full additive subcategory of an additive category closed under isomorphisms and direct summands. For any additive category $\mathcal{D}$, we denote by ${\mathcal{D}}(A, B)$ the set of morphisms from $A$ to $B$ in $\mathcal{D}$.

Let us briefly recall some definitions and basic properties of extriangulated categories from \cite{NP}. We omit some details here, but the reader can find them in \cite{NP}.

Assume that $\mathbb{E}: \mathcal{C}^{\rm op}\times \mathcal{C}\rightarrow {\rm Ab}$ is an additive bifunctor, where $\mathcal{C}$ is an additive category and ${\rm Ab}$ is the category of abelian groups. For any objects $A, C\in\mathcal{C}$, an element $\delta\in \mathbb{E}(C,A)$ is called an $\mathbb{E}$-extension.
Let $\mathfrak{s}$ be a correspondence which associates an equivalence class $$\mathfrak{s}(\delta)=\xymatrix@C=0.8cm{[A\ar[r]^x
 &B\ar[r]^y&C]}$$ to any $\mathbb{E}$-extension $\delta\in\mathbb{E}(C, A)$. This $\mathfrak{s}$ is called a {\it realization} of $\mathbb{E}$, if it makes the diagram in \cite[Definition 2.9]{NP} commutative.
 A triplet $(\mathcal{C}, \mathbb{E}, \mathfrak{s})$ is called an {\it extriangulated category} if it satisfies the following conditions.
\begin{enumerate}
\item $\mathbb{E}\colon\mathcal{C}^{\rm op}\times \mathcal{C}\rightarrow \rm{Ab}$ is an additive bifunctor.

\item $\mathfrak{s}$ is an additive realization of $\mathbb{E}$.

\item $\mathbb{E}$ and $\mathfrak{s}$ satisfy certain axioms in \cite[Definition 2.12]{NP}.
\end{enumerate}

In particular, we recall the following axioms which will be used later:

\emph{(ET4)} Let $\delta\in\mathbb{E}(D,A)$ and $\delta'\in\mathbb{E}(F, B)$ be $\mathbb{E}$-extensions realized by
 \begin{center} $\xymatrix{A\ar[r]^f&B\ar[r]^{f'}&D}$ and $\xymatrix{B\ar[r]^g&C\ar[r]^{g'}&F}$\end{center}
 respectively. Then there exists an object $E\in\mathcal{C}$, a commutative diagram
 $$\xymatrix{A\ar[r]^f\ar@{=}[d]&B\ar[r]^{f'}\ar[d]_g&D\ar[d]^d\\
A\ar[r]^h&C\ar[r]^{h'}\ar[d]_{g'}&E\ar[d]^e\\
&F\ar@{=}[r]&F}$$
in $\mathcal{C}$, and an $\mathbb{E}$-extension $\delta^{''}\in\mathbb{E}(E, A)$ realized by $\xymatrix{A\ar[r]^h&C\ar[r]^{h'}&E,}$
which satisfy the following compatibilities.

\begin{enumerate}
\item[(i)] $\xymatrix{D\ar[r]^d&E\ar[r]^{e}&F}$ realizes $f'_*\delta'$,

\item[(ii)] $d^*\delta^{''}=\delta$,

\item[(iii)] $f_*\delta^{''}=e^*\delta'$.
\end{enumerate}

\emph{(ET4)$^{\rm op}$} Dual of (ET4).

\begin{rem}
Note that both exact categories and triangulated categories are extriangulated categories $($see \cite[Example 2.13]{NP}$)$ and extension-closed subcategories of extriangulated categories are
again extriangulated $($see \cite[Remark 2.18]{NP}$)$. Moreover, there exist extriangulated categories which
are neither exact categories nor triangulated categories $($see \cite[Proposition 3.30]{NP}, \cite[Example 4.14]{ZZ} and \cite[Remark 3.3]{HZZ}$)$.
\end{rem}

Assume that $(\mathcal{C}, \mathbb{E}, \mathfrak{s})$ is an extriangulated category. By the Yoneda Lemma, any $\mathbb{E}$-extension $\delta\in \mathbb{E}(C, A)$ induces  natural transformations
 $\delta_\sharp: \mathcal{C}(-, C)\Rightarrow \mathbb{E}(-, A)$ $~~ ~$and$~~~$ $\delta^\sharp: \mathcal{C}(A, -)\Rightarrow \mathbb{E}(C, -)$.
For any $X\in\mathcal{C}$, these $(\delta_\sharp)_X$ and $\delta^\sharp_X$ are given as follows:

(1) $(\delta_\sharp)_X: \mathcal{C}(X, C)\rightarrow \mathbb{E}(X, A); ~f\mapsto f^*\delta.$

(2) $\delta^\sharp_X: \mathcal{C}(A, X)\rightarrow \mathbb{E}(C, X); ~g\mapsto g_*\delta.$
\begin{lem}\label{lem2} {\rm \cite[Corollary 3.12]{NP}}  Let $(\mathcal{C}, \mathbb{E}, \mathfrak{s})$ be an extriangulated category and $$\xymatrix@C=2em{A\ar[r]^{x}&B\ar[r]^{y}&C\ar@{-->}[r]^{\delta}&}$$ an $\mathbb{E}$-triangle. Then we have the following long exact sequences:

$\xymatrix@C=1cm{\mathcal{C}(C, -)\ar[r]^{\mathcal{C}(y, -)}&\mathcal{C}(B, -)\ar[r]^{\mathcal{C}(x, -)}&\mathcal{C}(A, -)\ar[r]^{\delta^\sharp}&\mathbb{E}(C, -)\ar[r]^{\mathbb{E}(y, -)}&\mathbb{E}(B, -)\ar[r]^{\mathbb{E}(x, -)}&\mathbb{E}(A, -),}$

$\xymatrix@C=1cm{\mathcal{C}(-, A)\ar[r]^{\mathcal{C}(-, x)}&\mathcal{C}(-, B)\ar[r]^{\mathcal{C}(-, y)}&\mathcal{C}(-, C)\ar[r]^{\delta_\sharp}&\mathbb{E}(-, A)\ar[r]^{\mathbb{E}(-, x)}&\mathbb{E}(-, B)\ar[r]^{\mathbb{E}(-, y)}&\mathbb{E}(-, C).}$

\end{lem}


Suppose that $(\mathcal{C}, \mathbb{E}, \mathfrak{s})$ is an extriangulated category with enough projectives and injectives.
If $\xymatrix{A\ar[r]&P\ar[r]&C\ar@{-->}[r]&}$ is an $\E$-triangle with $P\in Proj(\mathcal{C})$,  then $A$ is called the \emph{syzygy} of $C$ and is denoted by $\Omega(C)$. Similarly, for any $\E$-triangle $\xymatrix{A\ar[r]&I\ar[r]&C\ar@{-->}[r]&}$ with $I\in Inj(\mathcal{C})$, $C$ is called the \emph{cosyzygy} of $A$ and is denoted by $\Sigma(A)$.

For a subcategory $\mathcal{B}\subseteq \mathcal{C}$, put $\Omega^{0}\mathcal{B}=\mathcal{B}$, and for $i>0$, we define $\Omega^{i}\mathcal{B}$ inductively to be the subcategory consisting of syzygies of objects in $\Omega^{i-1}\mathcal{B}$, i.e.,
$$\Omega^{i}\mathcal{B}=\Omega(\Omega^{i-1}\mathcal{B}).$$
We call $\Omega^{i}\mathcal{B}$ the $i$-th syzygy of $\mathcal{B}$. Dually we define the $i$-th cosyzygy $\Sigma^{i}\mathcal{B}$ by $\Sigma^{0}\mathcal{B}=\mathcal{B}$ and $\Sigma^{i}\mathcal{B}=\Sigma(\Sigma^{i-1}\mathcal{B})$ for $i>0.$

In \cite{LN}, Liu and Nakaoka defined higher extension groups in an extriangulated category having enough projectives and injectives as $$\E^{i+1}(X, Y)\cong\E(X, \Sigma^{i}Y)\cong\E(\Omega^{i}X, Y)$$ for $i\geq 0$, and they showed the following result:

\begin{lem}\label{lem4} \cite[Proposition 5.2]{LN} Let $\xymatrix{A\ar[r]^{x}&B\ar[r]^{y}&C\ar@{-->}[r]^{\delta}&}$ be an $\E$-triangle. For any object $X\in \mathcal{B}$, there are long exact sequences
$$\xymatrix@C=0.5cm{\cdots\ar[r] &\E^{i}(X, A)\ar[r]^{x_{*}}&\E^{i}(X, B)\ar[r]^{y_{*}}&\E^{i}(X, C)\ar[r]&\E^{i+1}(X, A)\ar[r]^{x_{*}}&\E^{i+1}(X, B)\ar[r]^{y_{*}}&\cdots(i\geq1),}$$
$$\xymatrix@C=0.5cm{\cdots\ar[r] &\E^{i}(C, X)\ar[r]^{y^{*}}&\E^{i}(B, X)\ar[r]^{x^{*}}&\E^{i}(A, X)\ar[r]&\E^{i+1}(C, X)\ar[r]^{y^{*}}&\E^{i+1}(B, X)\ar[r]^{x^{*}}&\cdots(i\geq1).}$$
\end{lem}

{\bf From now on to the end of the paper, we always suppose that $\mathcal{C}:=(\mathcal{C}, \mathbb{E}, \mathfrak{s})$ is an extriangulated category with enough projectives and injectives.}

\section{\bf Computation of coresolution dimensions}\label{coresolution-dimension}
This section is devoted to preparations for the proof of our main results in this paper.
First, we introduce the notion of coresolution dimensions of objects in $\mathcal{C}$, and then
give some criteria for computing coresolution dimensions relative to different subcategories in $\mathcal{C}$.

\begin{df} Let $\mathcal{X}$ be a subcategory of $\mathcal{C}$.
\begin{enumerate}
\item For any non-negative integer $n$, we denote by $\widecheck{\X_{n}}$ (resp., $\widehat{\X_{n}}$) the class of objects $C\in \C$ such that there exists an $\E$-triangle sequence ~\begin{center}$C\rightarrow X_{0}\rightarrow\cdots\rightarrow X_{-n+1}\rightarrow X_{-n}$ (resp., $X_{n}\rightarrow X_{n-1}\rightarrow\cdots\rightarrow X_{0}\rightarrow C$)
    \end{center} with each $X_{i}\in \mathcal{X}$. Moreover, we set $$~\widecheck{\X}=\bigcup\limits_{n=0}^\infty \widecheck{\X_{n}}, \ \widehat{\X}=\bigcup\limits_{n=0}^\infty \widehat{{\X_n}}.$$

\item For any $C\in \mathcal{C}$, the \emph{$\mathcal{X}$-coresolution dimension} of $C$ is defined as
\vspace{2mm}
\begin{center}
coresdim$_{\mathcal{X}}(C)$ := min$\{n\in\mathbb{N}:C\in\w{\mathcal{X}_{n}}\}$.
\end{center}
\vspace{2mm}
If $C\not\in\w{\X_{n}}$ for any $n\in \mathbb{N}$, then coresdim$_{\X}(C)=\infty$.
\end{enumerate}
\end{df}

Recall from the introduction that $$\mathcal{X}^{\perp}=\{Y\in\mathcal{C} \ | \ \E^{i}(X, Y)=0 \mathrm{~for~all~} i\geq 1, \mathrm{~and~all~} X \in \mathcal{X}\}$$ for any subcategory $\mathcal{X}$ of $\mathcal{C}$. Similarly, we can define $^{\perp}\mathcal{X}.$ The following definition is an extriangulated analog of \cite[Definition 5.1]{MSSS1}.
\begin{df}
Let $\mathcal{X}$ and $\mathcal{W}$ be two subcategories of $\mathcal{C}$. We say that
\begin{enumerate}
\item $\mathcal{W}$ is a \emph{generator} for $\mathcal{X}$, if $\mathcal{W}\subseteq\mathcal{X}$ and for each object $X\in \mathcal{X}$, there exists an $\E$-triangle $\xymatrix@C=0.6cm{X'\ar[r]^{}&W\ar[r]^{}&X\ar@{-->}[r]^{}&}$ with $W\in \W$ and $X^{'}\in \X$.

\item $\mathcal{W}$ is \emph{$\mathcal{X}$-projective} if $\W\subseteq{^{\perp}\mathcal{X}}$.

\item $\mathcal{W}$ is an \emph{$\mathcal{X}$-projective generator} for $\X$ if $\mathcal{W}$ is a generator for $\mathcal{X}$ and $\W\subseteq{^{\perp}\mathcal{X}}$.
\end{enumerate}
\end{df}


\begin{ex}(1) Assume that $\C=R$-$\mathrm{Mod}$ is the category of left $R$-modules for a ring $R$. A left $R$-module $N$ is called \emph{Gorenstein projective} \cite{EJ2,Holm}~if there is an exact sequence of projective left $R$-modules
$$\mathbf{P}= \cdots\rightarrow P_{1}\rightarrow P_{0}\rightarrow P_{-1}\rightarrow P_{-2}\rightarrow\cdots$$
with $N\cong\mathrm{Ker}(P_{-1}\rightarrow P_{-2})$ such that $\mathrm{Hom}_{R}(\mathbf{P},Q)$ is exact for any projective left $R$-module $Q$. Let $\mathcal{GP}(R)$ be the subcategory of $R$-$\mathrm{Mod}$ consisting of all Gorenstein projective modules and $\mathcal{P}(R)$ the subcategory of  $R$-$\mathrm{Mod}$ consisting of all projective modules. Then $\mathcal{P}(R)$ is a $\mathcal{GP}(R)$-projective generator for $\mathcal{GP}(R)$.

(2) Let $\C$ be a triangulated category and $\mathcal{M}$ a silting subcategory. Then $\mathcal{M}$ is an $\mathcal{M}^{\perp}$-projective generator by \cite[Lemma 2.6]{DLWW}.

(3) Recall from \cite{ZZ1} that a subcategory $\mathcal{X}\subseteq\mathcal{C}$ is called  \emph{tilting} if $\mathcal{X}$ is a generator for $\mathcal{X}^{\perp}$ and $\mathcal{X}\subseteq\widehat{\mathcal{H}}$, where $\mathcal{H}=Proj(\mathcal{C})$ is the class of projective objects in $\mathcal{C}$. It follows that $\mathcal{X}$ is an $\mathcal{X}^{\perp}$-projective generator for $ \mathcal{X}^{\perp}$.
\end{ex}

In the following, let $\X$ and $\W$ be two subcategories of $\C$ such that $\W\subseteq\X$.

\begin{lem}\label{lem6}
Suppose that $\mathcal{X}$ is an extension-closed subcategory of $\mathcal{C}$. Consider two $\E$-triangles
$\xymatrix@C=0.6cm{N\ar[r]^{a}&X_{1}\ar[r]^{b}&D\ar@{-->}[r]^{}&}$ and   $\xymatrix@C=0.6cm{D\ar[r]^{c}&X_{0}\ar[r]^{d}&M\ar@{-->}[r]^{}&}$
 in $\mathcal{C}$ with $X_{0}, X_{1}\in \mathcal{X}$. If $\mathcal{W}$ is a generator for $\mathcal{X}$, then there exist two $\E$-triangles
$\xymatrix@C=0.6cm{N\ar[r]^{}&X_{1}'\ar[r]^{}&D'\ar@{-->}[r]^{}&}$ and $\xymatrix@C=0.6cm{D'\ar[r]^{}&W_{0}\ar[r]^{}&M\ar@{-->}[r]^{}&}$
 with $W_{0}\in \mathcal{W}$ and $X_{1}'\in\mathcal{X}$.
\end{lem}

\begin{proof}
 Note that $\mathcal{W}$ is a generator for $\mathcal{X}$. Then we have an $\E$-triangle $\xymatrix@C=0.6cm{X\ar[r]^{}&W_{0}\ar[r]^{}&X_{0}\ar@{-->}[r]^{}&}$with $W_{0}\in\W$ and $X\in\X$. By $(ET4)^{op}$, we obtain a commutative diagram in $\C$
$$\xymatrix{X\ar[r]^{}\ar@{=}[d]&D^{'}\ar[r]^{}\ar[d]_{}&D\ar[d]^{}&\\
X\ar[r]^{}&W_{0}\ar[r]^{}\ar[d]_{}&X_{0}\ar[d]^{}&\\
&M\ar@{=}[r]&M.}$$
 By \cite[Proposition 3.15]{NP}, we have the following commutative diagram in $\C$

$$\xymatrix{
&N\ar[d]_{} \ar@{=}[r] & N \ar[d]^{} \\
  X \ar@{=}[d] \ar[r]^{} & X_{1}' \ar[d]_{} \ar[r]^{} & X_{1}\ar[d]^{} \\
  X \ar[r]^{} & D'\ar[r]^{} & D.}$$
Since $\X$ is closed under extensions, it follows that $X_{1}'\in \X$. By the above two diagrams, we have the desired $\E$-triangles.
\end{proof}

\begin{lem}\label{lem7}
Suppose that $\X$ is closed under extensions and $\W$ is a generator for $\X$. Then for any $C\in\X$ and non-negative integer $n$, $C\in\w{\X_{n}}$ if and only if there exists an $\E$-triangle sequence
$$C\rightarrow X_{0}\rightarrow W_{-1}\rightarrow\cdots \rightarrow W_{-n+1}\rightarrow W_{-n}$$ with $X_{0}\in \X$ and $W_{-i}\in\W$ for $1\leq i\leq n$.
\end{lem}
\begin{proof}
 The ``if'' part is trivial. We prove the ``only if'' part by induction on $n$. If $n=1$, then there exist $\E$-triangles
  $\xymatrix@C=0.6cm{C\ar[r]^{}&X_{0}\ar[r]^{}&X_{-1}\ar@{-->}[r]^{}&}$ and $\xymatrix@C=0.6cm{X_{-1}\ar[r]^{}&X_{-1}\ar[r]^{}&0\ar@{-->}[r]^{}&}$ with $X_{-1}, X_0\in \X$.
   By Lemma \ref{lem6}, we obtain an $\E$-triangle $\xymatrix@C=0.6cm{C\ar[r]^{}&X_{0}'\ar[r]^{}&W_{-1}\ar@{-->}[r]^{}&}$ with $X_{0}'\in \X$ and $W_{-1}\in \W$, as desired.

 Suppose now $n\geq 2$. Then we have an $\E$-triangle $\xymatrix@C=0.6cm{C\ar[r]^{}&X_{0}'\ar[r]^{}&K\ar@{-->}[r]^{}&}$
with $X_{0}' \in \X$ and $K \in \w{\X_{n-1}}$. By induction, there exists an $\E$-triangle sequence
 $K\rightarrow X_{-1}\stackrel{f} \longrightarrow W_{-2}\rightarrow\cdots\rightarrow W_{-n}$ with $X_{-1} \in \X$ and
 $W_{-i} \in \W$ for $2\leq i\leq n$.
  Hence we have an $\E$-triangle sequence $K^{'}\stackrel{f_{1}}{\longrightarrow} W_{-2}\rightarrow \cdots \rightarrow W_{-n}$ and an $\E$-triangle $\xymatrix@C=0.6cm{K\ar[r]&X_{-1}\ar[r]^{f_2}&K'\ar@{-->}[r]^{}&}$ with $f_1f_2=f$.
 Applying Lemma \ref{lem6} to $\E$-triangles
 $\xymatrix@C=0.6cm{C\ar[r]^{}&X_{0}'\ar[r]^{}&K\ar@{-->}[r]^{}&}$ and $\xymatrix@C=0.6cm{K\ar[r]&X_{-1}\ar[r]^{f_2}&K'\ar@{-->}[r]^{}&,}$
  we obtain two $\E$-triangles
  $\xymatrix@C=0.6cm{C\ar[r]^{}&X_0\ar[r]&K''\ar@{-->}[r]^{}&}$ and $\xymatrix@C=0.6cm{K''\ar[r]^{}&W_{-1}\ar[r]^{}&K'\ar@{-->}[r]^{}&}$ with $X_{0}\in \X$ and $W_{-1}\in\W$.
   Therefore, we have an $\E$-triangle sequence
 $C\rightarrow X_{0}\rightarrow W_{-1}\rightarrow\cdots \rightarrow W_{-n+1}\rightarrow W_{-n}$ with $X_{0}\in \X$ and $W_{-i}\in\W$ for $1\leq i\leq n$.
\end{proof}

%
%
The following proposition is an extriangulated analog of \cite[Theorem 2.2]{DLWW} and \cite[Theorem 5.4]{MSSS1}. It shows that any object in $\w{\X}$ admits two $\E$-triangles: one giving rise to a left $\X$-approximation and the other to a right $\w{\W}$-approximation.
\begin{prop}\label{thm2}
Suppose that $\X$ is closed under extensions and $\W$ is a generator for $\X$. Consider the following conditions:
\begin{enumerate}
\item $C$ is in $\w{\X_n}.$

\item There exists an $\E$-triangle
$\xymatrix{C\ar[r]^{\psi^{C}}&X^{C}\ar[r]^{}&Y^{C}\ar@{-->}[r]^{\theta}&}$
with $X^{C}\in \X$ and $Y^{C}\in \w{\W_{n-1}}.$

\item There exists an $\E$-triangle
$\xymatrix{X_{C}\ar[r]^{}&Y_{C}\ar[r]^{\varphi_{C}}&C\ar@{-->}[r]^{\delta}&}$
with $X_{C}\in \X$ and $Y_{C}\in \w{\W_{n}}.$

\end{enumerate}
Then $(1)\Leftrightarrow(2)\Rightarrow(3)$. If $\X$ is also closed under cones of inflations, then $(3)\Rightarrow(2)$, and hence all three conditions are equivalent. If $\W$ is $\X$-projective, then $\psi^{C}$ is a left $\X$-approximation of $C$ and  $\varphi_{C}$ is a right $\w{\W}$-approximation of $C$.

\end{prop}
\begin{proof}
$(1)\Leftrightarrow(2)$ follows from Lemma \ref{lem7}.

$(2)\Rightarrow(3)$. Since $X^{C}\in\X$ and $\W$ is a generator for $\X$, we have an $\E$-triangle $$\xymatrix@C=0.5cm{X_{C}\ar[r]^{}&W\ar[r]^{}&X^{C}\ar@{-->}[r]^{}&}$$ with $W\in\W$ and $X_{C}\in\X$.
By $(ET4)^{op}$, we obtain a commutative diagram
$$\xymatrix{X_{C}\ar[r]^{}\ar@{=}[d]&Y_{C}\ar[r]^{}\ar[d]_{}&C\ar[d]^{}&\\
X_{C}\ar[r]^{}&W\ar[r]^{}\ar[d]_{}&X^{C}\ar[d]^{}&\\
&Y^{C}\ar@{=}[r]&Y^{C}.}$$
From the second column, it follows that $Y_{C}\in \w{\W_{n}}$. Hence the first row is the desired one.

Suppose that $\X$ is also closed under cones of inflations. We show $(3)\Rightarrow(2)$. Since $Y_{C}\in \w{\W_{n}}$, we have an $\E$-triangle
$\xymatrix@C=0.6cm{Y_{C}\ar[r]^{}&W\ar[r]^{}&Y^{C}\ar@{-->}[r]^{}&}$
with $W\in \W$ and $Y^{C}\in \w{\W_{n-1}}$. By $(ET4)$, we obtain a commutative diagram
$$\xymatrix{X_{C}\ar[r]^{}\ar@{=}[d]&Y_{C}\ar[r]^{}\ar[d]_{}&C\ar[d]^{}&\\
X_{C}\ar[r]^{}&W\ar[r]^{}\ar[d]_{}&X^{C}\ar[d]^{}&\\
&Y^{C}\ar@{=}[r]&Y^{C}.}$$ Since $\X$ is closed under cones of inflations, it follows that $X^{C}\in \X$. Hence the third column is the desired $\E$-triangle.

Assume that $\W$ is $\X$-projective. Applying $\mathcal{C}(-,X)$ to the $\E$-triangle
$$\xymatrix{C\ar[r]^{\psi^{C}}&X^{C}\ar[r]^{}&Y^{C}\ar@{-->}[r]^{\theta}&}$$
for any $X\in\X$, we have an exact sequence
$\xymatrix@C=0.6cm{\mathcal{C}(X^{C}, X)\ar[rr]^{\mathcal{C}(\psi^{C},X)}&&{\mathcal{C}}(C, X)\ar[r]^{}&\E(Y^{C}, X).}$
 Since $\E^{i}(\W,X)=0$ for all $i\geq1$, it is easy to see that $\E(Y^{C}, X)=0$. It follows that $\C(\psi^{C}, C)$ is an epimorphism, hence $\psi^{C}$ is a left $\X$-approximation of $C$. Similarly, we can prove that $\varphi_{C}$ is a right $\w{\W}$-approximation of $C$.
\end{proof}

%
%

\begin{cor}\label{prop2}
Let $\X$ be closed under extensions such that $\W$ is an $\X$-projective generator for $\X$. Then $\w{\X}$ is closed under extensions.
\end{cor}
\begin{proof} Suppose that $\xymatrix{A\ar[r]^{d}&B\ar[r]^{e}&C\ar@{-->}[r]^{\delta}&}$ is an $\E$-triangle with $A$ and $C$ in $\w{\X}$. We proceed by induction on $n:=$~coresdim$_{\X}(A)$.
 Assume $n=0$, which means that $A$ is in $\X$. As $C$ is in $\w{\X}$, there exists an $\E$-triangle
$\xymatrix@C=0.6cm{C\ar[r]^{p}&X^{C}\ar[r]^{}&Y^{C}\ar@{-->}[r]^{}&}$ with $X^{C}\in \X$ and $Y^{C}\in\w{\W}$ by Proposition \ref{thm2}.
 Since $\W\subseteq{^{\bot}\X}$, it is easy to see that $\w{\W}\subseteq {^{\perp}\X}$. It follows that $\E(p, A): \E(X^{C}, A)\rightarrow \E(C, A)$ is an isomorphism. Hence we obtain a commutative diagram
$$\xymatrix@C=2em{A\ar[r]^d\ar@{=}[d]&B\ar[r]^e\ar[d]&C\ar@{-->}[r]^{\delta}\ar[d]^{p}&\\
  A\ar[r]&Z\ar[r]&X^{C}\ar@{-->}[r]^{\theta}&
  }$$ with $p^{*}\theta=\delta$.
   By $(ET4)^{op}$, we have the commutative diagram
$$\xymatrix{A\ar[r]^{}\ar@{=}[d]&B\ar[r]^{}\ar[d]_{}&C\ar[d]^{p}&\\
A\ar[r]^{}&Z\ar[r]^{}\ar[d]_{}&X^{C}\ar[d]^{}&\\
&Y^{C}\ar@{=}[r]&Y^{C}.}$$
Since $X^{C}$ and $A$ are in $\X$, the object $Z\in \X$ as $\X$ is closed under extensions.
 Note that $Y^{C}\in \w{\W}$, hence in $\w{\X}$, it follows that $B\in \w{\X}$.

Assume that $n>0$ and $\xymatrix@C=0.6cm{A\ar[r]^{}&X^{0}\ar[r]^{}&L\ar@{-->}[r]^{}&}$ is an $\E$-triangle with $X^{0}\in{\X}$ and coresdim$_{\X}(L)=n-1$. By the dual of {\cite[Proposition 3.15]{NP}}, we have a commutative diagram
$$\xymatrix{
     A\ar[d]_{} \ar[r]^{} & B \ar[d]^{}\ar[r]^{}&C\ar@{=}[d] \\
  X^{0} \ar[d]_{} \ar[r]^{} & V' \ar[d]^{} \ar[r]^{} & C \\
  L \ar@{=}[r] & L.}
  $$
Since $X_{0}\in\X$, it follows that $\E(p,X^{0}): \E(X^{C}, X^{0})\rightarrow \E(C, X^{0})$ is an isomorphism.
 Hence we have a commutative diagram
 $$\xymatrix@C=2em{X^{0}\ar[r]^{}\ar@{=}[d]&V'\ar[r]^{}\ar[d]&C\ar@{-->}[r]^{p^{*}\theta_{1}}\ar[d]^{p}&\\
  X^{0}\ar[r]&V\ar[r]&X^{C}\ar@{-->}[r]^{\theta_{1}}&
  .}$$
 Using $(ET4)^{op}$, we obtain a commutative diagram
  $$\xymatrix{X^{0}\ar[r]^{}\ar@{=}[d]&V'\ar[r]^{}\ar[d]_{}&C\ar[d]^{p}&\\
X^{0}\ar[r]^{}&V\ar[r]^{}\ar[d]_{}&X^{C}\ar[d]^{}&\\
&Y^{C}\ar@{=}[r]&Y^{C}.&}$$
By $(ET4)$, we also have the following commutative diagram
 $$\xymatrix{B\ar[r]^{}\ar@{=}[d]&V'\ar[r]^{}\ar[d]_{}&L\ar[d]^{}&\\
B\ar[r]^{}&V\ar[r]^{}\ar[d]_{}&U\ar[d]^{}&\\
&Y^{C}\ar@{=}[r]&Y^{C}.}$$
 By the induction hypothesis, $U$ is in $\w{\X}$. Since $X^{C}$ and $X^{0}$ are in $\X$, $V$ belongs to $\X$, as $\X$ is closed under extensions. It follows that $B\in \w{\X}$.
 \end{proof}

\begin{lem}\label{lem9}
Suppose that $\X$ is closed under extensions and cones of inflations, and $\W$ is an $\X$-projective generator for $\X$. Given an $\E$-triangle
$\xymatrix{K\ar[r]^{x}&X\ar[r]^{y}&C\ar@{-->}[r]^{\delta}&}$ with $X\in\X$. Then $C\in \w{\X}$ if and only if $K\in \w{\X}$.
\end{lem}
\begin{proof} By definition, if $C\in \w{\X}$, then so is $K$. Now assume that $K\in \w{\X}$. By Proposition \ref{thm2}, we have an $\E$-triangle
$\xymatrix{K\ar[r]^{p}&X^{K}\ar[r]^{}&Y^{K}\ar@{-->}[r]^{\theta}&}$ with $X^{K}\in \X$ and $Y^{K}\in \w{\W}$. Note that $p:K\rightarrow X^{K}$ is a left $\X$-approximation of $K$.
Hence we obtain a commutative diagram
$$\xymatrix@C=2em{K\ar[r]^p\ar@{=}[d]&X^{K}\ar[r]\ar[d]^{f}&Y^{K}\ar@{-->}[r]^{\theta}\ar[d]^{g}&\\
 K\ar[r]&X\ar[r]&C\ar@{-->}[r]^{\delta}&
  }$$ with $g^*\delta=\theta$. Since  $Y^{K}\in \w{\W}$, there exists an $\E$-triangle
$\xymatrix{Y^{K}\ar[r]^{m}&W\ar[r]^{e}&L\ar@{-->}[r]^{\delta'}&}$ with $W\in\W$ and $L\in \w{\W}$.
 By \cite[Proposition 1.20]{LN}, we also have a commutative diagram
$$\xymatrix@C=2em{Y^{K}\ar[r]^m\ar[d]^g&W\ar[r]^e\ar[d]^h&L\ar@{-->}[r]^{\delta'}\ar@{=}[d]&\\
  C\ar[r]^u&V\ar[r]^v&L\ar@{-->}[r]^{g_*\delta'}&
  }$$ such that $\xymatrix@C=1,2cm{Y^{K}\ar[r]^{\tiny\begin{bmatrix}g\\m\end{bmatrix}\ \ \ }&C\oplus W\ar[r]^{\tiny\ \ \begin{bmatrix}-u&h\end{bmatrix}}&V\ar@{-->}[r]^{v^{*}\delta'}&}$ is an $\E$-triangle.
   By $(ET4)^{op}$, we ${\mathrm{obtain}}$ a commutative diagram
$$\xymatrix@C=4em{
  K\ar[r]^{} \ar@{=}[d] &H \ar[r]^{}\ar[d]& Y^{K} \ar[d]_{\tiny\begin{bmatrix}g\\m\end{bmatrix}} \ar@{-->}[r]^{g^{*}\delta} &  \\
  K\ar[r]^{\tiny\begin{bmatrix}x\\0\end{bmatrix}} & X\oplus W \ar[r]^{\tiny\begin{bmatrix}y&0\\0&1\end{bmatrix}} \ar[d]^{} \ar[d]^{} & C\oplus W\ar@{-->}[r]^{\tiny\begin{bmatrix}1&0\end{bmatrix}^{*}\delta} \ar[d]_{\tiny\begin{bmatrix}-u&h\end{bmatrix}} &\\
  & V \ar@{-->}[d]^{} \ar@{=}[r] &V\ar@{-->}[d]^{v^*\delta'}&\\
  &&&}\eqno{\raisebox{-8.5ex}{$(\ddagger)$}}
$$
where the second row can be obtained by the  proof similar to that of \cite[Lemma 3.7(2)]{HZZ}. Note that $g^{*}\delta=\theta$. Then we have the following commutative diagram
$$\xymatrix@C=2em{K\ar[r]^{}\ar@{=}[d]&H\ar[r]^{}\ar[d]^{}&Y^{K}\ar@{-->}[r]^{g^{*}\delta}\ar@{=}[d]&\\
  K\ar[r]^{p}&X^{K}\ar[r]^{}&Y^{K}\ar@{-->}[r]^{\theta}&
  .}$$ It follows that $H$ is isomorphic to $X^{K}$. Hence we can replace $H$ by $X^{K}$ in $(\ddagger)$.
   Since $W$ and $X$ are in $\X$, it follows that $V$ is in $\X$ as $\X$ is closed under cones of inflations. By Corollary \ref{prop2}, we have $C\oplus W\in \w{\X}.$ So there exists an $\E$-triangle
$\xymatrix{C\oplus W\ar[r]&X^{C\oplus W}\ar[r]^{}&Y^{C\oplus W}\ar@{-->}[r]^{}&}$ with $X^{C\oplus W}\in\X$ and $Y^{C\oplus W}\in\w{\W}$ by Proposition \ref{thm2}. By $(ET4)$, there is a commutative diagram

$$\xymatrix@C=2em{C\ar[r]^{}\ar@{=}[d]&C\oplus W\ar[r]^{}\ar[d]_{}&W\ar[d]^{}&\\
C\ar[r]^{}&X^{C\oplus W}\ar[r]^{}\ar[d]_{}&Z\ar[d]^{}&\\
&Y^{C\oplus W}\ar@{=}[r]&Y^{C\oplus W}&}$$
which implies that $Z\in \w{\X}$ as $Y^{C\oplus W}$ and $W$ are in $\w{\X}$. So $C\in \w{\X}$, as desired.
\end{proof}

Now we give characterizations of $\X$-coresolution dimensions of objects in $\w{\X}$.
\begin{prop}\label{thm3}
Suppose that $\X$ is closed under extensions and cones of inflations, and $\W$ is an $\X$-projective generator for $\X$. The following statements are equivalent for any $C\in \w{\X}$ and non-negative integer $n$.
\begin{enumerate}
\item {\rm coresdim}$_{\X}(C)\leq n$.

\item If $C\rightarrow X_{0}\rightarrow\cdots\rightarrow X_{-n+1}\rightarrow U$ is an $\E$-triangle sequence with $X_{-i}\in \X$ for $0\leq i\leq n-1$, then $U\in \X$.

\item $\E^{n+i}(Y, C)=0$ for any object $Y\in \w{\W}$ and all $i\geq1$.

\item $\E^{n+i}(W, C)=0$ for any object $W\in \W$ and all $i\geq1$.
\end{enumerate}
\end{prop}
\begin{proof} The proof is modeled on that of  \cite[Lemma 4.2]{MSSS1} and \cite[Lemma 4.7]{MSSS2}.

 $(3)\Rightarrow(4)$ and $(2)\Rightarrow(1)$ are trivial.

$(1)\Rightarrow(3).$ By Lemma \ref{lem7}, we have an $\E$-triangle sequence
$C\rightarrow X_{0}\rightarrow W_{-1}\rightarrow \cdots \rightarrow W_{-n+1}\rightarrow W_{-n}$ with $X_{0}\in \X$ and $W_{-i}\in\W$ for $1\leq i\leq n$.
Since $\W\subseteq {^{\perp}\X}$, it is easy to see that $Y\in{^{\perp}\X}$ for any $Y\in \w{\W}$.
Thus we have $\E^{n+i}(Y,C)\cong \E^{i}(Y,W_{-n})=0$ for all $i\geq 1$.

$(4)\Rightarrow(2).$ If
$C\rightarrow X_{0}\rightarrow\cdots\rightarrow X_{-n+1}\rightarrow U$ is an $\E$-triangle sequence with $X_{-i}\in \X$ for $0\leq i\leq n-1$. We have $\E^{i}(\W,U)\cong\E^{n+i}(\W, C)=0$ for all $i\geq 1$. Note that $U\in \w{\X}$ by Lemma \ref{lem9}. Hence there exists an $\E$-triangle
$\xymatrix@C=0.6cm{U\ar[r]^{}&X^{U}\ar[r]^{}&Y^{U}\ar@{-->}[r]^{}&}$ with $X^{U}\in \X$ and $Y^{U}\in \w{\W}$ by Proposition \ref{thm2}.
Since $\E^{i}(\W, U)=0$ for all $i\geq1$, it is easy to see that $\E(\w{\W},U)=0$.
Thus the above $\E$-triangle splits. So $U\in \X$ and $(2)$ follows.
\end{proof}

\section{\bf Proof of Theorem \ref{thm5}}\label{proof-of-main-result}

In this section, we first give a method to construct cotorsion pairs on the extriangulated category $\w{\X}$ for some extension closed subcategory $\X$ of an extriangulated category $\C$, and then establish relations between silting subcategories in the stable category of a Frobenius extriangulated category $\C$ and certain covariantly finite subcategories in $\C$ linked by the cotorsion pairs constructed above.
\begin{df} \emph{\cite[Definition 4.1]{NP}}\label{df:cotorsion pair}
Let $\mathcal{U}$, $\mathcal{V}$ $\subseteq$ $\mathcal{C}$ be a pair of full additive subcategories closed
under isomorphisms and direct summands. The pair ($\mathcal{U}$, $\mathcal{V}$) is called a {\it cotorsion
pair} on $\mathcal{C}$ if it satisfies the following conditions:
\begin{enumerate}
\item $\mathbb{E}(\mathcal{U}, \mathcal{V})=0$;

\item For any $C \in{\mathcal{C}}$, there exists a conflation $V^{C}\rightarrow U^{C}\rightarrow C$ satisfying
$U^{C}\in{\mathcal{U}}$ and $V^{C}\in{\mathcal{V}}$;

\item For any $C \in{\mathcal{C}}$, there exists a conflation $C\rightarrow V_{C} \rightarrow U_{C}$ satisfying
$U_{C}\in{\mathcal{U}}$ and $V_{C}\in{\mathcal{V}}$.
\end{enumerate}

In this case, put $\mathcal{I}=\mathcal{U}\bigcap\mathcal{V}$, and call it the \emph{core} of ($\mathcal{U}$, $\mathcal{V}$).
\end{df}

The following proposition is an extriangulated analog of the dual of \cite[Prposition 5.2]{MSSS1} and \cite[Theorem 3.5(b)]{MSSS2}.

\begin{prop}\label{prop1}
Let $\W\subseteq\X$ be two subcategories of $\C$ such that $\W$ is $\X$-projecive. Then the following statements hold.
\begin{enumerate}
\item $\w{\W}$ is $\X$-projective.

\item If $\W$ is a generator for $\X$, then $\W=\X\bigcap{^{\perp}\X}=\X\bigcap\w{\W}$.

\item If $\W$ is a generator for $\X$, then $\w{\W}=\w{\X}\bigcap{^{\perp}\X}$.
\end{enumerate}
\end{prop}
\begin{proof}
(1). Let $C\in \w{\W}$. Then there exists an $\E$-triangle sequence
$$C\rightarrow W_{0}\rightarrow W_{-1}\rightarrow\cdots \rightarrow W_{-n+1}\rightarrow W_{-n}$$ for some $n\geq 0$ with $W_{-i}\in\W$ for $1\leq i\leq n$.
It follows that $\E^{i}(C,\X)\cong\E^{n+i}(W_{-n}, \X)=0$ for all $i\geq 1$.
Hence $\w{\W}$ is $\X$-projective.

(2). It is obvious that $\W\subseteq\X\bigcap{^{\bot}\X}$. Let $X\in \X\bigcap{^{\perp}\X}$. Then we have an $\E$-triangle
$\xymatrix@C=0.6cm{X'\ar[r]^{}&W\ar[r]^{}&X\ar@{-->}[r]^{}&}$ with $X'\in \X$ and $W\in\W$ as $\W$ is a generator for $\X$. Moreover, $X\in {^{\bot}\X}$ implies that the above $\E$-triangle splits and so $X\in \W$. Hence $\W=\X\bigcap{^{\perp}\X}$.
 On the other hand, it is easy to see that $\W\subseteq\X\bigcap\w{\W}$. Since $\w{\W}$ is $\X$-projective by (1), it follows that $\X\bigcap\w{\W}\subseteq \X\bigcap{^{\perp}\X}$. Hence $\W=\X\bigcap\w{\W}.$

(3). Let $C\in \w{\X}\bigcap{^{\perp}\X}$. Then we have an $\E$-triangle
$\xymatrix@C=0.6cm{C\ar[r]^{}&X^{C}\ar[r]^{}&Y^{C}\ar@{-->}[r]^{}&}$ with $Y^{C}\in \w{\W}$ and $X^{C}\in\X$ by Proposition \ref{thm2}. Since $Y^{C}$ and $C$ are in ${^{\perp}\X}$, it follows that $X^{C}\in \X\bigcap{^{\perp}\X}$. Hence $X^{C}\in \W$ by $(2)$, implying $\w{\X}\bigcap{^{\bot}\X}\subseteq\w{\W}$.
 On the other hand, it is obvious that $\w{\W}\subseteq\w{\X}\bigcap{^{\perp}\X}.$ So $\w{\W}=\w{\X}\bigcap{^{\perp}\X}$.
\end{proof}

\begin{prop}\label{prop4}
Let $\X$ be closed under extensions such that $\W$ is an $\X$-projective generator for $\X$. If $\X$ is closed under cones of inflations, then $\w{\X}$ is closed under direct summands.
\end{prop}
\begin{proof}

Suppose $C_{1}\oplus C_{2}\in \w{\X}$.
  We proceed by induction on $n=$~coresdim$_{\X}(C_{1}\oplus C_{2})$. If $n=0$, then $C_{1}$ and $C_{2}$ are in $\X$.

Suppose $n>0$. There is an $\E$-triangle
$\xymatrix@C=3em{C_{1}\oplus C_{2}\ar[r]^{\tiny\begin{bmatrix}y_{1}&y_{2}\end{bmatrix}}&X\ar[r]^{}&K\ar@{-->}[r]^{\delta}&}$
with $X\in \X$ and coresdim$_{\X}(K)=n-1$. By $(ET4)$, we obtain the following commutative diagram

$$\xymatrix@C=3em{
  C_{1}\ar[r]^{\tiny\begin{bmatrix}1\\0\end{bmatrix}} \ar@{=}[d] &C_{1}\oplus C_{2} \ar[r]^{\tiny\begin{bmatrix}0&1\end{bmatrix}} \ar[d]^{\tiny\begin{bmatrix}y_{1}&y_{2}\end{bmatrix}}& C_{2} \ar[d]^{} \ar@{-->}[r]^{} &  \\
  C_{1}\ar[r]^{y_{1}} & X\ar[r]^{x_{1}} \ar[d]^{} & L_{1}\ar@{-->}[r]^{\delta_{1}} \ar[d]^{} &\\
  & K \ar@{-->}[d]^{\delta} \ar@{=}[r] &K\ar@{-->}[d]^{}&~ \\
  &&~~~&~}$$
  Similarly, we can obtain an $\E$-triangle $\xymatrix{C_{2}\ar[r]^{y_{2}}&X\ar[r]^{x_{2}}&L_{2}\ar@{-->}[r]^{\delta_{2}}&.}$
Hence there is an $\E$-triangle $$\xymatrix@C=3.5em{C_{1}\oplus C_{2}\ar[r]^{\tiny\begin{bmatrix}y_{1}&0\\0&y_{2}\end{bmatrix}}&X\oplus X\ar[r]^{\tiny\begin{bmatrix}x_{1}&0\\0&x_{2}\end{bmatrix}}&L_{1}\oplus L_{2}\ar@{-->}[r]^{\delta_{1}\oplus\delta_{2}}&.}$$
 By Lemma \ref{lem9}, $L_{1}\oplus L_{2}\in \w{\X}$, and Proposition \ref{thm3} shows that coresdim$_{\X}(L_{1}\oplus L_{2})\leq n-1$. By the induction hypothesis, $L_{1}$ and $L_{2}$ are in $\w{\X}$.
 It follows that $C_{1}$ and $C_{2}$ are in $\w{\X}$.
\end{proof}

The following corollary which can be viewed as an extriangulated analog of \cite[Theorem 3.5]{MSSS2} gives a method to construct cotorsion pairs on the extriangulated category $\w{\X}$.
\begin{cor}\label{corollary:4.4}
Let $\X$ be closed under extensions such that $\W$ is an $\X$-projective generator for $\X$. Then $\w{\X}$ is an extriangulated category. Moreover, if $\X$ is closed under cones of inflations, then $(\w{\W},\X)$ is a cotorsion pair on the extriangulated category $\w{\X}$.
\end{cor}
\begin{proof}
By Corollary \ref{prop2} and \cite[Remark 2.18]{NP}, we have that $\w{\X}$ is an extriangulated category. We only need to show that $\w{\W}$ is closed under direct summands by Proposition \ref{thm2}.
Since $\w{\W}=\w{\X}\bigcap{^{\perp}\X}$ by Proposition \ref{prop1}, it follows that $\w{\W}$ is closed under direct summands by Proposition \ref{prop4}. This completes the proof.
\end{proof}
\begin{lem}\label{lem10}
For a cotorsion pair $(\mathcal{U}, \mathcal{V})$ on $\C$, the following conditions are equivalent.

(1) $\mathcal{U}$ is resolving.

(2) $\mathcal{V}$ is coresolving.

(3) $\E^{i}(\mathcal{U},\mathcal{V})=0$ for all $i\geq 1$.

In this case, the cotorsion pair  $(\mathcal{U}, \mathcal{V})$ is called hereditary.
\end{lem}
\begin{proof}
The proof is similar to that of \cite[Lemma 5.24]{GT12}.
\end{proof}

%

Recall from the introduction that a subcategory $\HH\subseteq\C$ is called \emph{specially covariantly finite} in $\C$ provided that for any $C\in \C$, there is an $\E$-triangle  $\xymatrix@C=0.6cm{C\ar[r]^{}&H\ar[r]^{}&K\ar@{-->}[r]^{}&}$ such that $H\in\HH$ and $\E(K,\HH)=0$.

\begin{lem}\label{lem11}
Let $\HH$ be a subcategory of $\C$. Suppose that $\HH$ is coresolving and specially covariantly finite. Then $\HH\bigcap{^{\bot}\HH}$ is an $\HH$-projective generator for $\HH$.
\end{lem}
\begin{proof}
Let $H\in\HH$. Then there is an $\E$-triangle  $\xymatrix@C=0.6cm{X\ar[r]^{}&P\ar[r]^{}&H\ar@{-->}[r]^{}&}$ with $P\in Proj(\C)$ as $\C$ has enough projective objects. Since $\HH$ is specially covariantly finite, we have an $\E$-triangle $\xymatrix@C=0.6cm{X\ar[r]^{}&H'\ar[r]^{}&K\ar@{-->}[r]^{}&}$ with $H'\in\HH$ and $\E(K,\HH)=0$. By the dual of \cite[Proposition 3.15]{NP}, we have the commutative diagram
$$\xymatrix{
     X\ar[d]_{} \ar[r]^{} & P \ar[d]^{}\ar[r]^{}&H\ar@{=}[d] \\
  H^{'} \ar[d]_{} \ar[r]^{} & M \ar[d]^{} \ar[r]^{} & H \\
  K \ar@{=}[r] & K.}
  $$
Since $\HH$ is closed under extensions, it follows that $M\in \HH$.
 Note that $\E(P,\HH)=\E(K,\HH)=0$. Then one has $\E(M,\HH)=0$. We claim that $M\in \HH\bigcap{^{\bot}\HH}$. Indeed, for any positive integer $n$ and $H''\in\HH$, we have an $\E$-triangle sequence
$H''\rightarrow I_{0}\rightarrow I_{-1}\rightarrow\cdots\rightarrow I_{-n+1}\rightarrow L$ with $I_{-i}\in Inj(\C)$ for $0\leq i\leq n-1$. We have $\E^{n+1}(M,H'')\cong\E(M, L)=0$ as $\HH$ is coresolving and $\E(M,\HH)=0$. Hence $M\in \HH\bigcap{^{\bot}\HH}$. The second row in the above diagram implies that $\HH\bigcap{^{\bot}\HH}$ is an $\HH$-projective generator for $\HH$.
\end{proof}




When $\C$ is a Frobenius extriangulated category, the following theorem which is essential to the proof of Theorem \ref{thm5} gives a characterization of hereditary cotorsion pairs $(\mathcal{U}, \mathcal{V})$ on $\C$ with $\widehat{\mathcal{U}}=\C$ and $\widecheck{\mathcal{V}}=\C$.
\begin{thm}\label{thm4}
Let $\C$ be a Frobenius extriangulated category.
The assignments
\begin{center}
$(\mathcal{U}, \mathcal{V})\mapsto\mathcal{V}$ $~~$ and $~~$ $\mathcal{H}\mapsto(\w{\mathcal{M}},\mathcal{H})$,
\end{center}
 where $\mathcal{M}=\mathcal{H}\bigcap{^{\bot}\mathcal{H}}$, give mutually inverse bijections between the following classes:
\begin{enumerate}

\item Hereditary cotorsion pairs $(\mathcal{U}, \mathcal{V})$ on $\C$ with $\widehat{\mathcal{U}}=\widecheck{\mathcal{V}}=\C$.

\item Subcategories $\mathcal{H}$ of $\C$, which are specially covariantly finite and coresolving in $\C$ such that $\w{\mathcal{H}}=\C$ and for any $H\in \mathcal{H}$, there exists a positive integer $t\geq 1$ making $\E^{i}(H,\mathcal{H})=0$ for all $i\geq t$.
\end{enumerate}
\end{thm}

\begin{proof}
Let $(\mathcal{U}, \mathcal{V})$ be a hereditary cotorsion pair on $\C$ with $\widehat{\mathcal{U}}=\widecheck{\mathcal{V}}=\C$. Then $\mathcal{V}$ is coresolving in $\C$ by Lemma \ref{lem10}.

For any $V\in \mathcal{V}, V\in \widehat{\mathcal{U}}$ by assumption. Hence there exists a positive integer $n$ such that
$$U_{n}\rightarrow U_{n-1}\rightarrow\cdots\rightarrow U_{0}\rightarrow V$$
is an $\E$-triangle sequence with $U_{i}\in \mathcal{U}$ for $0\leq i\leq n$. We have $\E^{n+1}(V, \mathcal{V})\cong\E(U_{n}, \mathcal{V})=0$. Clearly, $n+1$ is the desired $t$.

Assume that $\mathcal{H}$ is a subcategory of $\C$ satisfying the condition in (2). By Lemma \ref{lem11}, $\mathcal{M}=\HH\bigcap{^{\bot}\HH}$ is an $\mathcal{H}$-projective generator for $\mathcal{H}$. By Proposition \ref{prop1}, $\w{\mathcal{M}}=\w{\mathcal{H}}\bigcap{^{\bot}\mathcal{H}}={^{\bot}\mathcal{H}}$, where the second equality is due to $\w{\mathcal{H}}=\C$. Therefore, $(\w{\mathcal{M}},\mathcal{H})$ is a hereditary cotorsion pair on $\C$ by Corollary \ref{corollary:4.4} and Lemma \ref{lem10}.

Let $X$ be an object in $\C$. Then we have an $\E$-triangle
$\xymatrix@C=0.6cm{L\ar[r]^{}&Y\ar[r]^{}&X\ar@{-->}[r]^{}&}$ with $Y\in\w{\mathcal{M}}$ and $L\in \mathcal{H}$. We claim that $L\in \tilde{\mathcal{M}}:=\widehat{(\check{\mathcal{M}})}$.
 Indeed, note that there exists a positive integer $t$ such that $\E^{i}(L,\mathcal{H})=0$ for all $i\geq t$. If $t=1$, then $L\in \w{\mathcal{M}}$ by the above $\E$-triangle, as desired. Assume $t>1$, then $\E^{t}(L,\mathcal{H})=\E(\Omega^{t-1}(L),\mathcal{H})=0$. Since $\mathcal{H}$ is coresolving, $\Omega^{t-1}(L)\in {^{\bot}\mathcal{H}}$. Meanwhile $\Omega^{t-1}(L)\in \w{\mathcal{H}}$ as $\mathcal{H}$ contains $Proj(\C)$. Hence $\Omega^{t-1}(L)\in \w{\mathcal{H}}\bigcap{^{\bot}\mathcal{H}}=\w{\mathcal{M}}$ by Proposition \ref{prop1}.
 Since $Proj(\C)\subseteq\mathcal{M}$, it follows that $L\in \tilde{\mathcal{M}}$. Hence $\C=\widehat{(\check{\mathcal{M}})}$, as desired.

Based on the above argument, it is enough to check that the compositions
\begin{center}
$(\mathcal{U}, \mathcal{V})\mapsto\mathcal{V}\mapsto (\w{\mathcal{V}\bigcap{^{\bot}\mathcal{V}}},\mathcal{V})$ $~~$ and $~~$ $\mathcal{H}\mapsto(\w{\mathcal{M}},\mathcal{H})\mapsto\mathcal{H}$
\end{center}
 are identities. Since $\mathcal{V}\bigcap{^{\bot}\mathcal{V}}$ is a $\mathcal{V}$-projective generator for $\mathcal{V}$,  $\w{\mathcal{V}\bigcap{^{\bot}\mathcal{V}}}=\w{\mathcal{V}}\bigcap
 {^{\bot}\mathcal{V}}=\C\bigcap\mathcal{U}=\mathcal{U}$, where the first equality is due to Proposition \ref{prop1} and the second equality is due to \cite[Remark 4.4]{NP}.
This completes the proof.
\end{proof}

Next we recall the triangulated structure of $\underline{\C}$ which is induced by a Frobenius extriangulated category $\C$~(see~\cite{ZZ}).

Let $\C$ be a Frobenius extriangulated category. For any $A\in \C$, we have an $\E$-triangle
$\xymatrix{A\ar[r]^{\alpha}&I\ar[r]^{\beta}&D\ar@{-->}[r]^{\delta}&,}$
where $I\in Inj(\C)$. Define $\mathbb{G}(A)=D$ to be the image of $D$ in $\underline{\C}$. Then $[1]:=\mathbb{G}:\underline{\C}\rightarrow \underline{\C}$ is an equivalence as an additive functor.

Let $\xymatrix{A\ar[r]^{x}&B\ar[r]^{y}&C\ar@{-->}[r]^{\theta}&}$ be an $\E$-triangle in $\C$. Then there exists the following commutative diagram
$$\xymatrix@C=2em{A\ar[r]^{x}\ar@{=}[d]&B\ar[r]^{y}\ar[d]^{b}&C\ar@{-->}[r]^{\theta}\ar[d]^{z}&\\
                  A\ar[r]^{\alpha}&I\ar[r]^{\beta}&A[1]\ar@{-->}[r]^{\delta}&
  }$$
as $I$ is injective. So we have the sequence
$\xymatrix@C=0.6cm{A\ar[r]^{\underline{x}}&B\ar[r]^{\underline{y}}&C\ar[r]^{\underline{z}}&A[1]}$
in $\underline{\C}$, which is called a standard triangle. The triangles in $\underline{\C}$ are defined as the sequences which are isomorphic to some standard triangle obtained in this way.

\begin{lem}\label{lem12}
Let  $(\mathcal{C}, \mathbb{E}, \mathfrak{s})$ be a Frobenius extriangulated category.

\begin{enumerate}
\item $X$ is a zero object in $\underline{\C}$ if and only if $X$ is an injective object in $\C$.
\item Let $X$ and $Y$ be in $\C$. Then $X \cong Y$ in $\underline{\C}$ if and only if there exist two injective objects $I$ and $Q$ such that $X\oplus I\cong Y\oplus Q$.
\end{enumerate}
\end{lem}
\begin{proof}(1). If $X$ is a zero object in $\underline{\C}$, then $\underline{\mathrm{Id}_{X}}=0$. So there exist $u:X\rightarrow Q$ and $h:Q\rightarrow X$ with $hu=\mathrm{Id}_{X}$ and $Q\in Inj(\C)$. Since $\C$ has enough injectives, there exists an inflation $f:X\rightarrow I$ with $I\in Inj(\C).$ Hence there exists $t:I\rightarrow Q$ with $tf=u$. Set $g:=ht$. Then $gf=htf=hu=\mathrm{Id}_{X}$.
 Therefore, we have a split $\E$-triangle $\xymatrix@C=0.6cm{X\ar[r]^{f}&I\ar[r]^{}&D\ar@{-->}[r]^{}&.}$ Thus $X$ is a direct summand of $I$, and hence $X$ is an injective object. The converse is clear.

                (2). Assume that we have $X\xrightarrow{\underline{f}}Y\xrightarrow{\underline{g}}X$ with $\underline{gf}=$$\underline{\mathrm{Id}_{X}}$ and $\underline{fg}=$$\underline{\mathrm{Id}_{Y}}$ in $\underline{\C}$. Then there exists a commutative diagram in $\C$
                $$\xymatrix@C=0.6cm{
  X \ar[rr]^{gf-\mathrm{Id}_{X}} \ar[dr]_{a}
                &  &    X     \\
                & Q\ar[ur]_{b}                 }$$ such that $Q$ is an injective object. So
                $\tiny\begin{bmatrix}g&-b\end{bmatrix}$$\tiny\begin{bmatrix}f\\a\end{bmatrix}$$=\mathrm{Id}_{X}$.
                Without loss of generality we may assume that $a$ is an inflation.  Thus $\tiny\begin{bmatrix}f\\a\end{bmatrix}$ is also an inflation by \cite[Corollary 3.16]{NP}. Hence we have an $\E$-triangle $$\xymatrix{X\ar[r]^{\tiny\begin{bmatrix}f\\a\end{bmatrix}}&Y\oplus Q\ar[r]^e&I\ar@{-->}[r]^{0}&.}$$
                 Thus we have the following commutative diagram
                $$\xymatrix@C=2.5em{X\ar[r]^{\tiny\begin{bmatrix}1\\0\end{bmatrix}}\ar@{=}[d]& X\oplus I\ar[r]^{\tiny\begin{bmatrix}
                0&1\end{bmatrix}}\ar[d]^h&I\ar@{-->}[r]^{0}\ar@{=}[d]&\\
  X\ar[r]^{\tiny\begin{bmatrix}f\\a\end{bmatrix}}&Y\oplus Q\ar[r]^e&I\ar@{-->}[r]^{0}&
  .}$$
 So $h:X\oplus I\rightarrow Y\oplus Q$ is an isomorphism and $\tiny\begin{bmatrix}f\\a\end{bmatrix}$$=h\tiny\begin{bmatrix}1\\0\end{bmatrix}$. Then $\underline{\tiny\begin{bmatrix}1\\0\end{bmatrix}}:X\rightarrow X\oplus I$ is an isomorphism in $\underline{\C}$ as $\underline{f}$ is an isomorphism.
  Therefore, there exists a morphism $\tiny\begin{bmatrix}m&n\end{bmatrix}:X\oplus I\ra X$ such that $\underline{\tiny\begin{bmatrix}1\\0\end{bmatrix}\tiny\begin{bmatrix}m&n\end{bmatrix}}=\tiny\begin{bmatrix}\underline{m}&\underline{n}\\0&0\end{bmatrix}=\tiny\begin{bmatrix}\underline{1}&0\\0&\underline{1}\end{bmatrix}.$
  It follows that $\underline{\mathrm{Id}_{I}}=0$ in $\underline{\C}$, then $I$ is injective in $\C$ by $(1)$. This completes the proof.
\end{proof}

\begin{lem}\label{lem13}
Let $\C$ be a Frobenius extriangulated category. If $A$ and $B$ are in $\C$, then $\E^{i}(A,B)\cong\underline{\C}(A,B[i])$ for all $i\geq 1$.
\end{lem}
\begin{proof}
By \cite[Lemma 3.8]{CZZ}, we have $\E(A,B)=\underline{\C}(A,B[1]).$ We proceed by induction. Note that we have $\E^{i}(A,B)\cong\E^{i-1}(\Omega A,B)$ and ${\underline{\C}}(\Omega A,B[i-1])\cong {\underline{\C}}(A,B[i])$ for all $i\geq 1$. Thus $\E^{i}(A,B)\cong\E^{i-1}(\Omega A,B)\cong{\underline{\C}}(\Omega A,B[i-1])\cong {\underline{\C}}(A,B[i])$ by the induction hypothesis.
\end{proof}

We recall the definition of a (bounded) co-t-structure.

\begin{definition} \cite{Bondarko,Pa}  {\rm Let $\mathcal{T}$ be a triangulated category. A \emph{co-t-structure} on $\mathcal{T}$ is a pair $(\mathcal{A},\mathcal{B})$ of subcategories of $\mathcal{T}$ such that
\begin{enumerate}
\item $\mathcal{A}[-1]\subseteq \mathcal{A}$ and $\mathcal{B}[1]\subseteq \mathcal{B}$.

\item $\mathrm{Hom}_{\mathcal{T}}(\mathcal{A}[-1],\mathcal{B})=0$.

\item Any object $T\in{\mathcal{T}}$ has a distinguished triangle
$$\xymatrix@C=0.6cm{X\ar[r]&T\ar[r]&Y\ar[r]&X[1],}$$
in $\mathcal{T}$ with $X\in{\mathcal{A}[-1]}$ and $Y\in{\mathcal{B}}$.
\end{enumerate}}
\end{definition}

Recall from \cite{MSSS2} that a co-t-structure $(\mathcal{A},\mathcal{B})$ is said to be \emph{bounded} if $\bigcup\limits_{n\in\mathbb{Z}}\mathcal{A}[n]=\mathcal{T}=\bigcup\limits_{n\in\mathbb{Z}}\mathcal{B}[n].$
Moreover, the bijective correspondence between bounded co-t-structures on a triangulted category $\mathcal{T}$ and silting subcategories of $\mathcal{T}$ has been shown by Mendoza Hern$\acute{\mathrm{a}}$ndez, S$\acute{\mathrm{a}}$enz, Santiago Vargas and Souto Salorio in \cite{MSSS2}.

\begin{thm}\cite[Corollary 5.9]{MSSS2}\label{thm:4.10} Let $\mathcal{T}$ be a triangulated category. The assignments
\begin{center}
 ${\mathcal{M}}\mapsto$~$(\widecheck{\mathcal{M}},{\mathcal{M}^{\bot}})$~$~~$ and $~~$~$(\mathcal{A},\mathcal{B})\mapsto{\mathcal{A}\bigcap{\mathcal{B}}}$
\end{center}
  give mutually inverse bijections between the following classes: 
\begin{enumerate}
\item Silting subcategories $\mathcal{M}$ of $\mathcal{T}$.
\item Bounded co-t-structures $(\mathcal{A},\mathcal{B})$ on $\mathcal{T}$.
\end{enumerate}
\end{thm}

\begin{lem}\label{lem14}\cite[Corollary 3.10]{CZZ}
Let $\C$ be a Frobenius extriangulated category. Suppose that $\mathcal{U}$ and $\mathcal{V}$ are two subcategories of $\C$ such that $Proj(\C)\subseteq \mathcal{U}\bigcap\mathcal{V}$. Then $(\mathcal{U}, \mathcal{V})$ is a cotorsion pair with the core $\mathcal{I}=\mathcal{U}\bigcap\mathcal{V}$ on $\C$ if and only if $(\mathcal{\underline{U}}, \mathcal{\underline{V}})$ is a cotorsion pair with the core $\underline{\mathcal{I}}$ on $\underline{\C}.$
\end{lem}
\begin{prop}\label{prop5}
Let $\C$ be a Frobenius extriangulated category. Suppose that $\mathcal{U}$ and $\mathcal{V}$ are two subcategories of $\C$ such that $Proj(\C)\subseteq \mathcal{U}\bigcap\mathcal{V}$.
 Then $(\mathcal{U}, \mathcal{V})$ is a hereditary cotorsion pair on $\C$ if and only if $(\mathcal{\underline{U}}, \mathcal{\underline{V}})$ is a co-t-structure on $\underline{\C}.$
\end{prop}
\begin{proof}
``$\Rightarrow$". By Lemma \ref{lem14} and the definition of co-t-structures, it suffices to check that $\mathcal{\underline{U}}$ is closed under $[-1]$. Let
$\xymatrix@C=0.6cm{U_{1}\ar[r]^{\underline{x}}&U_{2}\ar[r]^{\underline{y}}&U_{3}\ar[r]^{\underline{z}}&U_{1}[1]}$
be a triangle in $\underline{\C}$ with $U_{2}$ and $U_{3}$ in $\mathcal{\underline{U}}$. Then there exists a standard triangle $\xymatrix@C=0.6cm{U_{1}^{'}\ar[r]^{\underline{x^{'}}}&U_{2}^{'}\ar[r]^{\underline{y^{'}}}&U_{3}^{'}\ar[r]^{\underline{z^{'}}}&U_{1}^{'}[1]}$
which is isomorphic to the above triangle. Therefore $U_{i}^{'}\oplus P_{i}^{'}\cong U_{i}\oplus P_{i}$ by Lemma \ref{lem12}, where $P_{i}$ and $P_{i}^{'}$ are in $Proj(\C),\ i=1,2,3$.
Then $U_{2}^{'}$ and $U_{3}^{'}$
are in $\mathcal{U}$, it follows that $U_{1}^{'}$ is in $\mathcal{U}$ as $\mathcal{U}$ is resolving. So $U_{1}\in \mathcal{U}$. It follows that $\mathcal{\underline{U}}$ is closed under cocones. Note that $\mathcal{\underline{U}}$ is closed under extensions and cocones if and only if $\mathcal{\underline{U}}$ is closed under extensions and $[-1]$. It follows that $\mathcal{\underline{U}}$ is closed under $[-1]$. Hence
$(\mathcal{\underline{U}}, \mathcal{\underline{V}})$ is a co-t-structure on $\underline{\C}.$

``$\Leftarrow$". Let $\xymatrix@C=0.6cm{X\ar[r]^{}&U_{1}\ar[r]^{}&U_{2}\ar@{-->}[r]^{}&}$ be
an $\E$-triangle with $U_{1}$ and $U_{2}$ in $\mathcal{U}$.
Then $$\xymatrix@C=0.6cm{U_{2}[-1]\ar[r]&X\ar[r]^{}&U_{1}\ar[r]^{}&U_{2}&}$$
is a distinguished triangle in $\underline{\C}$.
Since $(\mathcal{\underline{U}}, \mathcal{\underline{V}})$ is a co-t-structure on $\underline{\C}$, one has $U_{2}[-1]\in \underline{\mathcal{U}}$. Thus $X\in \mathcal{\underline{U}}$,
and hence there exists $Y\in \mathcal{U}$ such that $X\oplus P\cong Y\oplus Q$, where $P$ and $Q$ are in $Proj(\C)$. It follows that $X\in \mathcal{U}$. So $(\mathcal{U}, \mathcal{V})$ is a hereditary cotorsion pair on $\C$ by Lemmas \ref{lem10} and \ref{lem14}.
\end{proof}

\begin{prop}\label{prop6}Let $\C$ be a Frobenius extriangulated category. Suppose that $\mathcal{U}$ and $\mathcal{V}$ are two subcategories of $\C$ such that $Proj(\C)\subseteq \mathcal{U}\bigcap\mathcal{V}$. Then $(\mathcal{U}, \mathcal{V})$ is a hereditary cotorsion pair on $\C$ with $\widehat{\mathcal{U}}=\widecheck{\mathcal{V}}=\C$ if and only if $(\mathcal{\underline{U}}, \mathcal{\underline{V}})$ is a bounded co-t-structure on $\underline{\C}.$
\end{prop}
\begin{proof}
By Proposition \ref{prop5}, one has that $(\mathcal{U}, \mathcal{V})$ is a hereditary cotorsion pair on $\C$ if and only if $(\mathcal{\underline{U}}, \mathcal{\underline{V}})$ is a co-t-structure on $\underline{\C}.$
 On the one hand, if $(\mathcal{U}, \mathcal{V})$ is a hereditary cotorsion pair on $\C$ with $\widehat{\mathcal{U}}=\widecheck{\mathcal{V}}=\C$, then $(\mathcal{\underline{U}}, \mathcal{\underline{V}})$ is a co-t-structure on $\underline{\C}$ with
$\widehat{(\underline{\mathcal{U}})}=\widecheck{(\underline{\mathcal{V}})}=\underline{\C}$. Hence $(\mathcal{\underline{U}}, \mathcal{\underline{V}})$ is a bounded co-t-structure on $\underline{\C}$ by \cite[Remark 4.2]{MSSS2}.
 On the other hand, if $(\mathcal{\underline{U}}, \mathcal{\underline{V}})$ is a bounded co-t-structure on $\underline{\C}$, then $\widehat{(\underline{\mathcal{U}})}=\widecheck{(\underline{\mathcal{V}})}=\underline{\C}$ by \cite[Remark 4.2]{MSSS2}. We claim that $\widehat{\mathcal{U}}=\widecheck{\mathcal{V}}=\C$.

 Take $C\in \C$. We proceed by induction on $n=\mathrm{coresdim}_{\underline{\mathcal{V}}}(C)$. If $n=0$, then $C\in\mathcal{V}$ by Lemma \ref{lem12}. Suppose $n>0$. There is a triangle
 $\xymatrix@C=0.6cm{C\ar[r]^{\underline{a}} & V \ar[r]^{\underline{b}} & X\ar[r]^{\underline{c}} & C[1]}$
 with $V\in \mathcal{V}$ and $\mathrm{coresdim}_{\underline{\mathcal{V}}}(X)=n-1$ in $\underline{\C}$.
 Therefore, we have a standard triangle
  $\xymatrix@C=0.6cm{C^{'}\ar[r]^{\underline{a}^{'}} & V^{'} \ar[r]^{\underline{b}^{'}} & X^{'}\ar[r]^{\underline{c}^{'}} & C^{'}[1]}$
 which is isomorphic to the above triangle.
  By the induction hypothesis, $X\in\w{\mathcal{V}}$. Since there is an isomorphism $X\oplus P\cong X^{'}\oplus P^{'}$ with $P,P^{'}\in Proj(\C)$ by Lemma \ref{lem12}, $X^{'}$ is in $\w{\mathcal{V}}$ by the proof that $\w{\mathcal{V}}$ is closed under direct summands in the proof of Proposition \ref{prop4}. Hence $C^{'}$ is in $\w{\mathcal{V}}$. Similarly, we have $C\in \w{\mathcal{V}}$. So $\w{\mathcal{V}}=\C$. Dually, we can show that $\widehat{\mathcal{U}}=\C$. This completes the proof.
\end{proof}

In the following, given a subcategory $\mathcal{M}$ of $\C$ with $Proj(\C)\subseteq\mathcal{M}$, the stable category of $\mathcal{M}$ respect to the subcategory $Proj(\C)$ will be denoted by $\underline{\mathcal{M}}$. Note that $\underline{\mathcal{M}}\subseteq\underline{\C}$, and we let
 $$\underline{\mathcal{M}}^{\bot}=\{N\in \underline{\C}\mid \underline{\C}(\underline{\mathcal{M}},N[i])=0 \ \text{for all~} i>0 \}$$
and
$$ ^{\bot}\underline{\mathcal{M}}=\{N\in \underline{\C}\mid \underline{\C}(N,\underline{\mathcal{M}}[i])=0\ \text{for all~} i>0 \}.$$

We are now in a position to prove Theorem \ref{thm5} from the introduction.

{\bf Proof of Theorem \ref{thm5}.} Assume that $\mathcal{H}$ is a subcategory satisfying the condition in $(2)$. Then one has $Proj(\C)\subseteq \mathcal{H}\bigcap {^{\bot}\mathcal{H}}$. It follows from Theorem \ref{thm4} that $(^{\perp}\mathcal{H}, \mathcal{H})$ is a hereditary cotorsion pair on $\C$ with $\widehat{(^{\perp}\mathcal{H})}=\widecheck{\mathcal{H}}=\C$. Since $\underline{\mathcal{H}}\bigcap\underline{^{\bot}\mathcal{H}}
=\underline{\mathcal{H}\bigcap{^{\bot}\mathcal{H}}}$ by Lemma \ref{lem14}, we get that $\underline{\mathcal{H}\bigcap{^{\bot}\mathcal{H}}}$ is a silting subcategory of $\underline{\C}$ with $Proj(\C)\subseteq \mathcal{H}\bigcap {^{\bot}\mathcal{H}}$ by Proposition \ref{prop6} and Theorem \ref{thm:4.10}. So $\Phi$ is well defined.

Let $\underline{\mathcal{M}}$  be a  silting subcategory of $\underline{\C}$ with $Proj(\C)\subseteq \mathcal{M}$.
 Then $(^{\bot}(\underline{\mathcal{M}}^{\bot}),\underline{\mathcal{M}}^{\bot})$ is a bounded co-t-structure on $\underline{\C}$ by Theorem \ref{thm:4.10}.
 We claim that $(^{\bot}(\underline{\mathcal{M}}^{\bot}),\underline{\mathcal{M}}^{\bot})=
 (\underline{^{\bot}(\mathcal{M}^{\bot})},\underline{\mathcal{M}^{\bot}}).$
Indeed, by Lemma \ref{lem13}, one has $\underline{\mathcal{M}}^{\bot}\subseteq \underline{\mathcal{M}^{\bot}}$.
 If $X\in \underline{\mathcal{M}^{\bot}}$, then there exists an object $Y$ in $\mathcal{M}^{\bot}$ with  $X\oplus P\cong Y\oplus Q$ by Lemma \ref{lem12}, where $P$ and $Q$ are in $Proj(\C)$.
 Hence $$\underline{\C}(\underline{\mathcal{M}},X[i])=\underline{\C}(\underline{\mathcal{M}},(X\oplus P)[i])=\underline{\C}(\underline{\mathcal{M}},(Y\oplus Q)[i])=\E^{i}(\mathcal{M},Y\oplus Q)=0,$$ where the third equality is due to Lemma \ref{lem13}. Thus $X\in {\underline{\mathcal{M}}^{\bot}}$.
 Next, we show that $^{\bot}(\underline{\mathcal{M}}^{\bot})=\underline{^{\bot}(\mathcal{M}^{\bot})}$.
 For $C\in \underline{^{\bot}(\mathcal{M}^{\bot})}$, there exists an object $Y$ in $^{\bot}(\mathcal{M}^{\bot})$ such that $C\oplus P_{1}\cong Y\oplus Q_{1}$ by Lemma \ref{lem12}, where $P_1$ and $Q_1$ are in $Proj(\C)$.
  Hence one has
  $$\underline{\C}(C,\underline{\mathcal{M}}^{\bot}[i])=
  \underline{\C}(C,\underline{\mathcal{M}^{\bot}}[i])=
  \underline{\C}(C\oplus P_{1},\underline{\mathcal{M}^{\bot}}[i])=
  \underline{\C}(Y\oplus Q_{1},\underline{\mathcal{M}^{\bot}}[i])=
  \E^{i}(Y\oplus Q_{1},\mathcal{M}^{\bot})=0$$ for all $i\geq 1$.
  Thus $\underline{^{\bot}(\mathcal{M}^{\bot})}\subseteq{^{\bot}(\underline{\mathcal{M}}^{\bot})}$.
  If $Y\in{^{\bot}(\underline{\mathcal{M}}^{\bot})}$, one has $$\E^{i}(Y,\mathcal{M}^{\bot})=\underline{\C}(Y,\underline{\mathcal{M}^{\bot}}[i])
 =\underline{\C}(Y,\underline{\mathcal{M}}^{\bot}[i])=0.$$
 Hence $Y\in \underline{^{\bot}(\mathcal{M}^{\bot})}$.
 So $^{\bot}(\underline{\mathcal{M}}^{\bot})=\underline{^{\bot}(\mathcal{M}^{\bot})}$.
 So ${\mathcal{M}^{\bot}}$ satisfies the condition of $(2)$ by Theorem \ref{thm4} and Proposition \ref{prop6}.

  Let $\mathcal{M}_{1}$ and $\mathcal{M}_{2}$ be two subcategories of $\C$ with  $Proj(\C)\subseteq \mathcal{M}_{1}\bigcap \mathcal{M}_{2}$.
  Assume that $\underline{\mathcal{M}_{1}}=\underline{\mathcal{M}_{2}}$.
  Indeed, for any $X\in\mathcal{M}_{1}$, we have $X\oplus P\cong Y\oplus Q$ by Lemma \ref{lem11}, where $Y$ is in $\mathcal{M}_{2}$, $P$ and $Q$ are in $Proj(\C)$. Thus $X^{\perp}= Y^{\perp}$ in $\C$, and hence ${\mathcal{M}_{1}}^{\perp}={\mathcal{M}_{2}}^{\perp}$ in $\C$.
  So $\Psi$ is well defined by the above argument.

  Let $\underline{\mathcal{M}}$  be a  silting subcategory of $\underline{\C}$ with $Proj(\C)\subseteq \mathcal{M}$. Then $\Phi\Psi(\underline{\mathcal{M}})=\Phi(\mathcal{M}^{\bot})=\underline{{\mathcal{M}^{\bot}}
  \bigcap {^{\bot}(\mathcal{M}^{\bot})}}$. By the above argument,
  $(^{\bot}(\underline{\mathcal{M}}^{\bot}),\underline{\mathcal{M}}^{\bot})=
 (\underline{^{\bot}(\mathcal{M}^{\bot})},\underline{\mathcal{M}^{\bot}})$ is a bounded co-t-structure.
 Hence, one has $$\underline{{\mathcal{M}^{\bot}}
  \bigcap {^{\bot}(\mathcal{M}^{\bot})}}=\underline{{\mathcal{M}^{\bot}}}
  \bigcap \underline{{^{\bot}(\mathcal{M}^{\bot})}}
  =\underline{\mathcal{M}}^{\bot}\bigcap(^{\bot}(\underline{\mathcal{M}}^{\bot})=\underline{\mathcal{M}},$$ where the first equality is due to Lemma \ref{lem14}, and the third equality is due to Theorem \ref{thm:4.10}. Thus $\Phi\Psi$ is an identity.  Assume that $\mathcal{H}$ is a subcategory satisfying the condition in $(2)$. Then $\Psi\Phi(\mathcal{H})=\Psi(\underline{\mathcal{H}\bigcap{^{\bot}\mathcal{H}}})
  ={(\mathcal{H}\bigcap{^{\bot}{\mathcal{H}})^{\bot}}}$.
Note that $(\w{\mathcal{M}},\mathcal{H})$ is a hereditary cotorsion pair with $\mathcal{M}=\mathcal{H}\bigcap{^{\bot}\mathcal{H}}$ by Theorem \ref{thm4}. It is easy to check that ${\mathcal{M}^{\bot}}={(\w{\mathcal{M}})^{\bot}}$. Thus ${(\mathcal{H}\bigcap{^{\bot}\mathcal{H}})^{\bot}}={\mathcal{M}^{\bot}}={(\w{\mathcal{M}})^{\bot}}
  =\mathcal{H}$, and hence $\Psi\Phi$ is an identity.
  This completes the proof.   \hfill$\Box$

Using Theorem \ref{thm5}, we have the following Auslander-Reiten type correspondence for silting subcategories over a triangulated category.
\begin{cor}\cite[Corollary 3.7]{DLWW} \label{cor:4.13}
Let $\mathcal{T}$ be a triangulated category. The assignments
\begin{center}
 $\mathcal{M}\mapsto$~$\mathcal{M}^{\bot}$~$~~$ and $~~$~$\mathcal{H}\mapsto\mathcal{H}\bigcap{^{\bot}\mathcal{H}}$
\end{center}
  give mutually inverse bijections between the following classes: 
\begin{enumerate}

\item Silting subcategories $\mathcal{M}$ of $\mathcal{T}$.
\item Subcategories $\mathcal{H}$ of $\mathcal{T}$, which are specially covariantly finite and coresolving in $\C$ such that $\w{\mathcal{H}}=\C$ and for any $H\in \mathcal{H}$, there exists a positive integer $t\geq 1$ making $\mathrm{Hom}_{\mathcal{T}}(H,\mathcal{H}[i])=0$ for all $i\geq t$.
\end{enumerate}
\end{cor}

\begin{remark}
Let $\mathcal{T}$ be a triangulated category. Then resolving subcategories are often called cosuspended or desuspended and coresolving subcategories are often called suspended {\rm(}see \cite{MSSS1}{\rm)}.
\end{remark}

%
%
%

In what follows, let $\mathcal{A}$ be an abelian category with enough projective objects.
We denote by $C(\mathcal{A})$ the category of complexes  and by $K(\mathcal{A})$ the homotopy category over $\mathcal{A}$.
Let $\mathcal{S}$ be the class of short exact sequences which split in each degree. Then $(C(\mathcal{A}),\mathcal{S})$ is a Frobenius exact category whose $\mathcal{S}$-projective objects are precisely the contractible complexes and $\underline{C(\mathcal{A})}=K(\mathcal{A})$ by \cite[Section 13.4]{B}.
Hence we get the following characterization for silting subcategories of $K(\mathcal{A})$.

\begin{cor}Consider the above Frobenius exact category $C(\mathcal{A})=(C(\mathcal{A}),\mathcal{S})$ and its stable category $K(\mathcal{A})=\underline{C(\mathcal{A})}$. The assignments
\begin{center}
 $\underline{\mathcal{M}}\mapsto$~${\mathcal{M}^{\bot}}$~$~~$ and $~~$~$\mathcal{H}\mapsto\underline{\mathcal{H}\bigcap{^{\bot}\mathcal{H}}}$
\end{center}
  give mutually inverse bijections between the following classes: 
\begin{enumerate}
\item Silting subcategories $\underline{\mathcal{M}}$ of $K(\mathcal{A})$ with $\mathcal{M}$ containing all contractible complexes.

\item Subcategories $\mathcal{H}$ of the Frobenius exact category $C(\mathcal{A})$, which are specially covariantly finite and coresolving in $C(\mathcal{A})$ such that $\w{\mathcal{H}}=C(\mathcal{A})$ and for any $H\in \mathcal{H}$, there exists a positive integer $t\geq 1$ making $\mathrm{Ext}_{C(\mathcal{A})}^{i}(H,\mathcal{H})=0$ for all $i\geq t$.

\end{enumerate}
\end{cor}

Let $\mathcal{A}$ be a Grothendieck category with enough projectives and $\mathcal{P}$ the subcategory of projective objects in $\mathcal{A}$. Recall from \cite{S1988} that a complex $P$ over $\mathcal{A}$ is called \emph{$K$-projective} if $\mathrm{Hom}_{K(\mathcal{A})}(P,E)=0$ for any exact complex $E$ over $\mathcal{A}$.
We denote by $C_{kproj}(\mathcal{A})$ the category of $K$-projective complexes over $\mathcal{A}$ and by $K_{kproj}(\mathcal{A})$ the category of complexes obtained from $C_{kproj}(\mathcal{A})$ by factoring out null-homotopic chain maps. Furthermore, we set $C_{dgproj}(\mathcal{A})=C_{kproj}(\mathcal{A})\bigcap C(\mathcal{P})$, where $C(\mathcal{P})$ is the category of complexes over $\mathcal{P}$. Note that every complex $X$ admits a quasi-isomorphism $f:P\rightarrow X$ with $P\in{C_{dgproj}(\mathcal{A})}$ by \cite[Corollary 3.5]{S1988}. It follows that $K_{kproj}(\mathcal{A})=D(\mathcal{A})$, where $D(\mathcal{A})$ is the unbounded derived category.  Note that $C_{dgproj}(\mathcal{A})$ is an extension-closed subcategory of $C(\mathcal{A})$.
Let $\mathcal{S}'$ be the class of short exact sequences whose all terms are in $C_{dgproj}(\mathcal{A})$.
Then $(C_{dgproj}(\mathcal{A}),\mathcal{S}')$ is a Frobenius exact category and $\underline{C_{dgproj}(\mathcal{A})}=K_{kproj}(\mathcal{A})$.
Thus $D(\mathcal{A})=\underline{C_{dgproj}(\mathcal{A})}$.
Hence we get the following characterization for silting subcategories of $D(\mathcal{A})$.
\begin{cor} Consider the above Frobenius exact category $(C_{dgproj}(\mathcal{A}),\mathcal{S}')$ and its stable category $D(\mathcal{A})=\underline{C_{dgproj}(\mathcal{A})}$.
 The assignments
\begin{center}
 $\underline{\mathcal{M}}\mapsto$~${\mathcal{M}^{\bot}}$~$~~$ and $~~$~$\mathcal{H}\mapsto\underline{\mathcal{H}\bigcap{^{\bot}\mathcal{H}}}$
\end{center}
  give mutually inverse bijections between the following classes: 

\begin{enumerate}
 \item Silting subcategories $\underline{\mathcal{M}}$ of $D(\mathcal{A})$ with $\mathcal{M}$ containing all contractible complexes whose all terms are in $\mathcal{P}$.

\item Subcategories $\mathcal{H}$ of $C_{dgproj}(\mathcal{A})$, which are specially covariantly finite and coresolving in $C_{dgproj}(\mathcal{A})$ such that $\w{\mathcal{H}}=C_{dgproj}(\mathcal{A})$ and for any $H\in \mathcal{H}$, there exists a positive integer $t\geq 1$ making $\mathrm{Ext}_{C(\mathcal{A})}^{i}(H,\mathcal{H})=0$ for all $i\geq t$.
 \end{enumerate}
\end{cor}

Let $R$ be a ring. Denote by $\mathcal{GP}(R)$ the class of Gorenstein projective left $R$-modules and by $\mathcal{P}(R)$ the class of projective left $R$-modules. Then $\mathcal{GP}(R)$ is a Frobenius exact category (see \cite[4.1]{K}).
%
%
%

\begin{cor} Consider the above Frobenius exact category $\mathcal{GP}(R)$ and its stable category $\underline{\mathcal{GP}(R)}$. The assignments
\begin{center}
 $\underline{\mathcal{M}}\mapsto$~${\mathcal{M}^{\bot}}$~$~~$ and $~~$~$\mathcal{H}\mapsto\underline{\mathcal{H}\bigcap{^{\bot}\mathcal{H}}}$
\end{center}
  give mutually inverse bijections between the following classes: 
\begin{enumerate}

\item Silting subcategories $\underline{\mathcal{M}}$ of $\underline{\mathcal{GP}(R)}$ with $\mathcal{P}(R)\subseteq\mathcal{M}$.

\item Subcategories $\mathcal{H}$ of $\mathcal{GP}(R)$, which are specially covariantly finite and coresolving in $\mathcal{GP}(R)$ such that $\w{\mathcal{H}}=\mathcal{GP}(R)$ and for any $H\in \mathcal{H}$, there exists a positive integer $t\geq 1$ making $\mathrm{Ext}_{R}^{i}(H,\mathcal{H})=0$ for all $i\geq t$.

 \end{enumerate}
\end{cor}

\renewcommand\refname
\textbf{Yajun Ma}\\
Department of Mathematics, Nanjing University, Nanjing 210093, China.\\
E-mail: \textsf{13919042158@163.com}\\[1mm]
\textbf{Nanqing Ding}\\
Department of Mathematics, Nanjing University, Nanjing 210093, China.\\
E-mail: \textsf{nqding@nju.edu.cn}\\
\textbf{Yafeng Zhang}\\
Department of Mathematics, Nanjing University, Nanjing 210093, China.\\
E-mail: \textsf{470985396@163.com}\\[1mm]
\textbf{Jiangsheng Hu}\\
Department of Mathematics, Jiangsu University of Technology,
 Changzhou 213001, China.\\
E-mail: \textsf{jiangshenghu@hotmail.com}
\end{document}